\documentclass[12pt]{amsart}

\usepackage{graphicx}
\usepackage{amssymb}
\usepackage{amsthm}
\usepackage{amsfonts}
\usepackage{amsmath}
\usepackage{amstext}

\usepackage{latexsym}
\usepackage{longtable}
\usepackage{rotating}
\usepackage{verbatim}
\usepackage{hyperref}
\usepackage{subfigure}
\usepackage{mathrsfs}
\usepackage{tensor}
\usepackage{color}
\usepackage[textsize=small]{todonotes}
\newcounter{todocounter}

\theoremstyle{plain}
\newtheorem{thm}{Theorem}[section]
\newtheorem*{main}{Main~Theorem}
\newtheorem*{mult}{Multiparameter~Theorem}
\newtheorem{cor}[thm]{Corollary}
\newtheorem{lem}[thm]{Lemma} 
\newtheorem{prop}[thm]{Proposition}

\theoremstyle{definition}
\newtheorem{defi}[thm]{Definition}
\theoremstyle{remark}
\newtheorem{rem}[thm]{Remark}
\numberwithin{equation}{section}

\newtheorem{example}[thm]{Example}

\newcommand{\C}{\mathbb{C}}
\newcommand{\N}{\mathbb{N}}
\newcommand{\R}{\mathbb{R}}

\newcommand{\Q}{\mathbb{Q}}
\newcommand{\Z}{\mathbb{Z}}

\newcommand{\proj}{\mathbb{P}}

\newcommand{\rk}{\overline K_1}
\newcommand{\lk}{\underline K_1}

\DeclareMathOperator{\ord}{ord_0}
\DeclareMathOperator{\size}{S}

\usepackage[all]{xy}
\newdir{ >}{{}*!/-10pt/@{>}}
\newdir^{ (}{{}*!/-5pt/@^{(}}
\newdir_{ (}{{}*!/-5pt/@_{(}}

\theoremstyle{plain}
\newtheorem*{lem*}{Lemma} 
\newtheorem*{claim*}{Claim}

\theoremstyle{remark}
\newtheorem*{rem*}{Remark}

\newcommand{\sr}[1]%
{\ifmmode{}^\dagger\else${}^\dagger$\fi\ifvmode
\vbox to 0pt{\vss
 \hbox to 0pt{\hskip\hsize\hskip1em
 \vbox{\hsize3cm\raggedright\pretolerance10000
 \noindent #1\hfill}\hss}\vss}\else
 \vadjust{\vbox to0pt{\vss%
 \hbox to 0pt{\hskip\hsize\hskip1em%
 \vbox{\hsize3cm\raggedright\pretolerance10000%
 \noindent #1\hfill}\hss}\vss}}\fi%
}

\def\p{\partial}
\let\on=\operatorname

\def\<{\langle}
\def\>{\rangle}
\renewcommand{\o}{\circ}

\def\al{\alpha}
\def\tal{\widetilde \alpha}

\def\ga{\gamma}
\def\de{\delta}
\def\ep{\epsilon}

\def\la{\lambda}

\def\si{\sigma}

\def\ta{\tau}

\def\ph{\phi}
\def\vh{\varphi}

\def\ps{\psi}
\def\om{\omega}

\def\De{\Delta}

\def\La{\Lambda}

\def\Ps{\Psi}
\def\Om{\Omega}

\def\C{\mathbb{C}}

\def\N{\mathbb{N}}
\def\Q{\mathbb{Q}}
\def\R{\mathbb{R}}
\def\Z{\mathbb{Z}}

\def\cA{\mathcal{A}}

\def\cC{\mathcal{C}}

\def\cL{\mathcal{L}}

\def\cN{\mathcal{N}}

\def\sD{\mathscr{D}}

\def\sI{\mathscr{I}}

\def\sK{\mathscr{K}}

\def\sfC{\mathsf{C}}

\def\sfE{\mathsf{E}}

\def\kk{{m}}
\def\lifta{{\widehat a}}
\def\cover{\mathcal {CV}}

\title[Regularity of roots of polynomials]
{Regularity of  roots of polynomials}

\author[Adam Parusi\'nski and  Armin Rainer]
{Adam Parusi\'nski and Armin Rainer}

\address {Adam Parusi\'nski: Univ. Nice Sophia Antipolis, CNRS,  LJAD, UMR 7351, 06108 Nice, France}

\email{adam.parusinski@unice.fr}

\address{Armin Rainer: Fakult\"at f\"ur Mathematik, Universit\"at Wien, 
Oskar-Morgenstern-Platz~1, A-1090 Wien, Austria}

\email{armin.rainer@univie.ac.at}

\begin{document}

\begin{abstract}
We show that smooth curves of monic complex polynomials $P_a (Z)=Z^n+\sum _{j=1}^n a_j Z^{n-j}$, $a_j : I \to \C$ with $I \subset \R$ a compact interval,   
have absolutely continuous roots in a uniform way. 
More precisely, there exists a positive integer $k$ and a rational number $p >1$, both depending only on the degree $n$, 
such that if $a_j \in C^{k}$ then any continuous choice of roots of $P_a$     
is absolutely continuous with derivatives in $L^q$ for all $1 \le q < p$, in a uniform way with respect to $\max_j\|a_j\|_{C^k}$. 
The uniformity allows us to deduce also a multiparameter version of this result.
The proof is based on formulas for the roots of the universal polynomial $P_a$ in terms of its coefficients $a_j$ 
which we derive using resolution of singularities. 
For cubic polynomials we compute the formulas as well as bounds for $k$ and $p$ explicitly. 
\end{abstract}

\thanks{Supported by the Austrian Science Fund (FWF), Grants P~22218-N13 and P 26735-N25, and by ANR project STAAVF (ANR-2011 BS01 009).}

\maketitle

\tableofcontents

\section*{Introduction}

This paper is dedicated to the solution of a basic problem in perturbation theory and differential analysis.
Given a monic polynomial with smooth coefficients (or a matrix with smooth entries) it is natural to ask for the 
optimal regularity of the roots (or of the eigenvalues). For instance, this question arises in finding  
local solutions of partial differential equations with multiple characteristics. 

In connection with the study of a class of pseudodifferential systems,  
Spagnolo \cite{Spagnolo00} asked whether a smooth ($C^\infty$) curve of monic complex polynomials admits a 
locally absolutely continuous parameterization of its roots. And if so, whether it is possible to choose the 
absolutely continuous roots uniformly with respect to the coefficients on compact subintervals. 
We answer these questions affirmatively and prove the following stronger result; see also Theorem~\ref{theorem}. 
Our proof builds on the recent result of Ghisi and Gobbino \cite{GhisiGobbino13} who found the optimal regularity of 
radicals of functions that we combine with the resolution of singularities.

\begin{main}
	For every $n\in \N_{>0}$ there is $k=k(n) \in \N_{>0}$ and  $p=p(n)>1$ such that the following holds.  
	Let $I\subset \R$ be a compact interval and let 
	\[
		P_{a(t)}(Z) = Z^n + \sum_{j=1}^n a_j(t) Z^{n-j} \in C^{k}( I)[Z] 	
	\] 
	be a monic polynomial with coefficients $a_j \in C^{k}( I)$, $j = 1,\ldots,n$.   
	\begin{enumerate}
		\item 
		Let $\la_j \in C^0(I)$, $j = 1,\ldots,n$, be a continuous parameterization of the roots of $P_a$ on $I$.
		Then the distributional derivative of each $\la_j$ in $I$ is a measurable function $\la_j' \in L^q(I)$ for every 
		$q \in [1,p)$. 
		In particular, each $\la_j \in W^{1,q}(I)$ for every $q \in[1,p)$.
		\item This regularity of the roots is uniform.  Let $\{P_{a_\nu} ; \nu \in \cN\}$,
		\[
			P_{a_\nu(t)}(Z) = Z^n + \sum_{j=1}^n a_{\nu,j}(t) Z^{n-j} \in C^{k}( I)[Z], ~\nu \in \cN, 	
		\] 
		be a family of curves of polynomials, indexed by $\nu$ in some set $\cN$, so that the set of coefficients 
		$\{a_{\nu,j} ; \nu \in \cN, j=1,\ldots,n\}$ is bounded in $C^{k}( I)$. 
		Then the set
		\[
			\qquad \{\la_{\nu} \in C^0(I);  P_{a_\nu}(\la_\nu)=0
			\text{ on $I$}, ~\nu \in \cN\}
		\]		
		is bounded in $W^{1,q}(I)$ for every $q \in [1,p)$.	\end{enumerate}	
\end{main}

$L^q$ denotes the Lebesgue space and $W^{1,q}$ the Sobolev space with respect to Lebesgue measure. 
We want to stress the fact that a continuous curve of monic complex polynomials $P_{a(t)}$, $t \in \R$, allows for 
a continuous parameterization of its roots. This is no longer true if the parameter space has more than one dimension 
due to monodromy. For multiparameter families of polynomials we obtain the following result; see also Theorem~\ref{multpar}.  

\begin{mult}
	Let $k=k(n) \in \N_{>0}$ and  $p=p(n)>1$ be as above.  
	Let $U \subset \R^m$ be open and let $P_{a(x)}(Z) \in C^{k}(U)[Z]$ 	
	be a monic polynomial with coefficients $a_j \in C^{k}(U)$, $j = 1,\ldots,n$.   
	\begin{enumerate}
		\item 
		Let $\la \in C^0(V)$ represent a root of $P_a$, i.e., $P_a(\la) = 0$, on a relatively compact open subset 
		$V \Subset U$. 
		Then the distributional gradient of $\la$ in $V$ is a measurable function $\nabla \la \in [L^q(V)]^m$ for every 
		$q \in [1,p)$. 
		In particular, $\la \in W^{1,q}(V)$ for every $q \in[1,p)$.
		\item The regularity of the roots is uniform. Let $\{P_{a_\nu} ; \nu \in \cN\}$
		be a family of polynomials, indexed by $\nu$ in some set $\cN$, so that the set of coefficients 
		$\{a_{\nu,j} ; \nu \in \cN, j=1,\ldots,n\}$ is bounded in $C^{k}(U)$. 
		Let $V \Subset U$. 
		Then the set
		\[
			\qquad \{\la_{\nu} \in C^0(V);  P_{a_\nu}(\la_\nu)=0
			\text{ on $V$}, ~\nu \in \cN\}
		\]			
		is bounded in $W^{1,q}(V)$ for every $q \in [1,p)$.	\end{enumerate}
\end{mult}

In \cite{Spagnolo00} Spagnolo proved that the pseudodifferential $n \times n$ system
\begin{align*} 
  u_t + i A(t,D_x) u + B(t,D_x) u = f(t,x), \qquad (t,x) \in I \times U,
\end{align*}
where $A(t,\xi)$, $B(t,\xi)$ are matrix symbols of order $1$ and $0$, respectively, and $A(t,\xi)$ is homogeneous 
of degree $1$ in $\xi$ for $|\xi|\ge 1$,  
is locally solvable in the Gevrey class $G^s$ for
$1 \le s \le n/(n-1)$  
and semi-globally solvable in $G^s$ for $1 < s < n/(n-1)$ under the following assumptions: 
the eigenvalues of $A(t,\xi)$ admit a parameterization $\ta_1(t,\xi),\ldots,\ta_n(t,\xi)$ such that each $\ta_j(t,\xi)$
is absolutely continuous in $t$, uniformly with respect to $\xi$, i.e.,
\begin{align*}
	\tag{$\cA_1$}  
	|\p_t \ta_j(t,\xi)| \le \mu(t,\xi) (1+|\xi|^2)^{\frac{1}{2}},
\quad  \text{ with } \mu(~,\xi) \text{ equi-integrable on } I,
\end{align*}
and for each $\xi$ the imaginary parts of the $\ta_j(t,\xi)$ do not change sign for varying $t$ and $j$, i.e., 
\begin{align*}
	\tag{$\cA_2$}  
	\forall \xi \quad \text{ either } \on{Im} \ta_j(t,\xi) \ge 0,\quad \forall t,j, 
	\quad \text{ or } \on{Im} \ta_j(t,\xi) \le 0,\quad \forall t,j.
\end{align*}
Our Main Theorem implies that the Assumption~\thetag{$\cA_1$} is automatically satisfied. 
Indeed, this follows by applying the Main Theorem to the characteristic polynomial of the matrix 
$(1+|\xi|^2)^{-\frac{1}{2}}A(t,\xi)$ and noting that the entries of $(1+|\xi|^2)^{-\frac{1}{2}}A(t,\xi)$ and its 
iterated partial derivatives with respect to $t$ are globally bounded in $\xi$, since $A(t,\xi)$ is a symbol of order 
$1$. 

Spagnolo formulated the removal of Assumption~\thetag{$\cA_1$} as an open problem in \cite{Spagnolo00}, p.~1122, and he 
successfully tackled the case of quadratic and cubic polynomials in \cite{Spagnolo99}. Note that the problems of smoothly 
choosing roots of polynomials on one hand and eigenvalues of \emph{arbitrary} quadratic matrices on the other hand 
are equivalent; whereas the perturbation theory for normal matrices is easier and allows for stronger results, 
cf.~Rainer \cite{RainerN} and references therein. 

We would like to remark that our result represents a complex analogue of Bronshtein's Theorem on the regularity of 
the roots of hyperbolic polynomials. A monic polynomial is called hyperbolic if all its roots are real. 
Bronshtein's Theorem, first proved in Bronshtein \cite{Bronshtein79} and generalizing the classical Glaeser inequality 
\cite{Glaeser63R}, states that any continuous parameterization of the 
roots of a hyperbolic polynomial of degree $n$ with $C^{n-1,1}$ coefficients is locally Lipschitz. It plays a crucial 
role for weakly hyperbolic Cauchy problems. Different proofs appeared in Wakabayashi \cite{Wakabayashi86} and 
in Parusi\'nski and Rainer \cite{ParusinskiRainerHyp}.  

In the absence of hyperbolicity the roots cannot fulfill a Lipschitz condition and in a certain sense absolute 
continuity is the best one can hope for; in fact the degree of summability $p$ tends to $1$ as $n$ goes to $\infty$.   
The first result towards absolute continuity of the roots 
is probably Lemma~1 in Colombini, Jannelli, and Spagnolo \cite{CJS83} which states that 
for a real-valued non-negative function $f$ of class $C^{k,\al}$ on a compact interval $I$, with 
$k \in \N_{\ge 1}$ and $0 \le \al \le 1$, the radical $f^{1/(k+\al)}$ is absolutely continuous on $I$ and satisfies
\[
\|(f^{\frac{1}{k+\al}})'\|_{L^1(I)}^{k+\al} \le C(k,\al,I) \|f\|_{C^{k,\al}(I)}.
\] 
Tarama \cite{Tarama00} extended this lemma to real-valued functions (not necessarily non-negative). 
A better summability for the weak partial derivatives of $f^{1/(k+1)}$ was obtained by Colombini and Lerner \cite{CL03} 
for non-negative $C^{k+1}$ functions $f$ of several real variables.

The case of radicals of functions was completely settled recently by Ghisi and Gobbino \cite{GhisiGobbino13} by 
finding their optimal regularity. They showed that, if $f$ is a real-valued continuous function 
and there exists $g \in C^{k,\al}(I)$ so that $|f|^{k+\al}=|g|$ on $I$, then   
$f' \in L_w^p(I)$, where $\frac{1}{p}+\frac{1}{k+\al}=1$, and
\[
\|f'\|_{p,w,I} \le 
C(k) \max\Big\{[\on{H\ddot{o}ld}_{\al,I}(g^{(k)})]^{\frac{1}{k+\al}} |I|^{\frac{1}{p}}, 
\|g'\|_{L^\infty(I)}^{\frac{1}{k+\al}}\Big\};
\]
in particular, $f \in W^{1,q}(I)$ for every $q \in [1,p)$.  
Here $L^p_w(I)$ denotes the weak Lebesgue space equipped with the quasi-norm 
$\|f\|_{p,w,I} := \sup_{r\ge 0} \big\{r \cdot \cL^1(\{t \in I ; |f(t)| > r\})^{\frac{1}{p}}\big\}$,
where $\cL^1$ is the one dimensional Lebesgue measure. 
By $\on{H\ddot{o}ld}_{\al,I}(g^{(k)})$ we mean the $\al$-H\"older constant of $g^{(k)}$ on $I$, and $|I|=\cL^1(I)$ is the 
length of the interval $I$. 
Ghisi and Gobbino also provided examples that show that the assumptions as well as the conclusion in their theorem are 
best possible. 
We use this result in a substantial way.

The mentioned results all treat special cases, where the algebraic structure of the polynomials is very simple: 
the roots are either given by  
radicals or can be expressed by radicals (by Cardano's formulas).  
A different approach was pursued in Rainer \cite{RainerAC}, where 
no restrictions on the algebraic structure of the polynomial were imposed. Under the assumption that no two roots meet 
with infinite order of contact it was shown that the roots of a $C^\infty$ curve of monic polynomials are locally 
absolutely continuous.
We also mention Rainer \cite{RainerQA}, where it was proved that the roots of a monic polynomial whose coefficients are 
functions in several variables that belong to any quasianalytic class satisfying some stability properties 
admit a parameterization by (special) functions of bounded variation.  

Our proof of the Main Theorem is based on formulas for the roots of the universal monic polynomial 
$P_a$ in terms of its 
coefficients $a=(a_1,\ldots,a_n)\in \C^n$. The derivation of these formulas represents the third major result of this paper; 
see Theorem~\ref{roots}. Using Hironaka's resolution of singularities \cite{Hironaka64}, we construct a tower of 
smooth principalizations 
$$M_1 = \C^n  \stackrel{ \sigma_2 }\longleftarrow M_2  \stackrel{ \sigma_{3,2} }\longleftarrow  M_3
\stackrel{ \sigma_{4,3} }\longleftarrow  \cdots \stackrel{ \sigma_{n,n-1} }\longleftarrow M_{n} $$
which successively principalize the generalized discriminant ideals 
$\sD_m \subset \C[a]$, $m=2,\ldots,n$, 
that encode the stratification of the space of polynomials by root multiplicity. 
In fact, the zero set of $\sD_m$ is exactly the set of those $a \in \C^n$ for which $P_a$ has at most $m-1$ distinct 
roots. 
We show that, locally on $M_{n}$, the roots of the pulled back polynomial $P_{\si_n^* (a)}$ are given by rational linear combinations 
$\sum_{m=1}^n A_m \, \vh_m \o \si_{n,m}$ where 
\[
	\varphi_\kk  = f_\kk^ {\al_\kk}  
	\psi_\kk (y_{\kk,1}^{1/q_\kk}, \ldots , y_{\kk, r_\kk}^{1/q_\kk}, y_{\kk,r_\kk+1}, \ldots , y_{\kk,n}).
\]
Here $\si_m = \si_2 \o \si_{3,2} \o \cdots \o \si_{m,m-1}$, $\si_{n,m} = \si_{m+1,m} \o \cdots \o \si_{n,n-1}$, 
$f_m \in \sD_m$ is a local generator of $\si_m^*(\sD_m)$, $\psi_\kk$ is a convergent power series, 
$q_m \in \N_{\ge 1}$, $\al_m \in \frac{1}{q_m} \N_{\ge 1}$,   
and 
$(y_{\kk,i})$ is a privileged system of local coordinates so that $f_m^{-1}(0)$ 
is given by $y_{\kk,1}\cdots y_{\kk,r_m}=0$ (cf.\ Subsection \ref{coordinates} and Definition \ref{localdata}).   
These formulas are  obtained in a natural way by a consecutive factorization procedure of  the pull-backs $P_{\si_i^* (a)}$, $i=2, \ldots, n$, so that each step contributes exactly to one summand.  
Thanks to these formulas we are able to reduce the problem to radicals of functions and use the result 
of Ghisi and Gobbino (cf.\ Lemma \ref{estimate}).    

The paper is divided into three parts. The first part presents the three main results of this paper. Section 
\ref{sec:formulas} is devoted to the formulation of the result on the formulas for the roots; see Theorem 
\ref{roots}. In Section \ref{sec:functspaces} we mainly fix notation on function spaces. 
The Main Theorem \ref{theorem} is proved in Section \ref{absolutecontinuity}, assuming validity of Theorem \ref{roots}, 
and the Multiparameter Theorem \ref{multpar} is deduced in Section \ref{sec:mult}.

The second part of the paper is dedicated to the proof of Theorem \ref{roots}. 
The strategy of the proof is briefly outlined in Section \ref{strategy}. 
In Section \ref{principality} we find a convenient criterion of principality of the ideals $\sD_m$. 
In Sections \ref{convexity} and \ref{splitting} we further develop the necessary tools utilized 
in the proof of Theorem \ref{roots} which is finally carried out in Section \ref{proofofmaintheorem}.

In the third part we illustrate our method of proof by discussing the case of cubic polynomials 
in detail. Here the resolution is explicit, and we can specify more precisely the degree of differentiability of 
the coefficients and the degree 
of summability of the derivative of the roots, 
namely $k(3) = 6$ and $p(3) = 6/5$.

\textbf{Acknowledgement.}  Part of the work was done while the second author enjoyed the hospitality of the 
mathematics department at the university of Nice.

\textbf{Notation and terminology.}  By a \emph{normal crossing} we mean a function that is locally equivalent to a monomial, i.e. equals a monomial times an analytic unit.   The zero set of an ideal $\sI$ will be denoted by $V(\sI )$.  For two real-valued functions $\varphi$ and $\psi$  we write  $\varphi \sim \psi$  if  there exists $C>0$ such that $\varphi \le C \psi$ and $\psi \le C\varphi$.   By  $\lceil x \rceil$ we denote the celling  function,  that is the least integer bigger than or equal to $x$.  


\part{Absolute continuity of roots}\label{part1}


\bigskip

\section{Formulas for the roots} \label{sec:formulas}

\subsection{Generalized discriminant ideals}  

Let 
\begin{align}\label{polynomial}
P_a (Z)=Z^n+\sum _{j=1}^n a_j Z^{n-j}
\end{align}
be a unitary polynomial with coefficients $a=(a_1, \ldots , a_n )\in \C^n$.   
We denote by $\xi (a) = \{ \xi_1 (a), \ldots , \xi _n(a)\}$ the unordered set of roots of $P_a$ 
and assign to $a_i$ the weight $i$ so that homogeneous permutation invariant polynomials in 
$\xi$ are precisely the weighted homogeneous polynomials in $a$.  

Let $N\in \N$ be a large constant fixed throughout the paper ($N\ge \max_{1\le s\le n} \binom n s$).  
For $2\le \kk \le n$ we denote by  $\sD_{N,\kk} $, or simply by  $\sD_{\kk} $,  the ideal 
of $\C[a]$ generated by all $f ^ {N!/ s} \in \C[a]$, where 
$f$ runs over all homogeneous polynomials in $\prod _{i\ne j \in I } (\xi _i -\xi_j)$, $I\subset \{1, \ldots , n\}$, 
$|I|=\kk$, of degree $s\le N$ in $\xi$, that are invariant by the permutations of $\xi_i$.

\begin{example}
$\sD_{N,n} $ is the principal ideal generated by the $N!/n(n-1)$-th power of the discriminant of $P_a$.
\end{example}

By replacing $Z$ by $Z- a_1/n$ we define a new polynomial 
$$P_{\hat a} (Z) = Z^n+ \hat a_2 Z^{n-2} + \cdots + \hat a_n:= P_a(Z - a_1/n).  $$ 
Each $\hat a_j$ is a weighted homogeneous polynomial in the $a_i$'s of the weighted degree $j$.  
This transformation $a\to \hat a$, called \emph{Tschirnhausen transformation}, shifts the roots of $P_a$ 
 by  $a_1/n$. Since the polynomials in $\prod _{i\ne j \in I } (\xi _i -\xi_j)$
are invariant by shifts of the roots, Tschirnhausen transformation does not change the ideal  $\sD_{\kk} $.   
Therefore, in what follows, we may suppose without loss of generality  that $P_a$ is in \emph{Tschirnhausen 
form} 
\begin{align}\label{Tschirnhausen}
P_a (Z)=Z^n+ a_2 Z^{n-2} + \cdots + a_n
\end{align}

\begin{prop}\label{firststep}
Suppose that $P_a$ is in Tschirnhausen form \eqref{Tschirnhausen}.  Then 
$\sD_{N,2}$ is the ideal of $\C[a_2, \ldots, a_n]$ generated by the weighted homogeneous polynomials 
in $a_2, \ldots, a_n$ of weighted degree $N!$.  
\end{prop}

\begin{proof}
If $P_a$ is in  Tschirnhausen form then $\sum_{i=1}^n \xi_i =0$ and hence  
$\xi_i = {\frac 1 n} \sum_j (\xi_i -\xi _j)$.  Therefore 
any polynomial in the $\xi_i$'s  is a polynomial in the $(\xi_i-\xi_j)$'s.  
\end{proof}

Thus the zero set $V(\sD_2)$ of $\sD_2$ is exactly the set of those $a$ for which $P_a$ has precisely 
one root (i.e. $a_2=\cdots= a_n=0$ if $P_a$ is in Tschirnhausen form \eqref{Tschirnhausen} and then 
this root is zero). 
In general, the zero set of $\sD_{\kk}$ consists of those $a\in \C^n$ for which $P_a$  has at most $\kk -1$ 
distinct roots, cf. Corollary \ref{zeroset} below.


\subsection{Smooth principalization of an ideal}  
Let $\sI=(f_1, \ldots, f_s)$ be an ideal of $\C [x]$, $x=(x_1, \ldots, x_m)$.  We shall assume 
$\sI\ne (0)$.  
The blowing-up of $\sI$, $\sigma : M=Bl_\sI \C^m \to \C^m$ can then be realized as follows, see for instance \cite{Hartshorne77} Example 7.17.2.  
The variety $M$ is the irreducible component of 
$$
 \{(x, y) \in \C^m \times \proj^{s-1}; f_i(x) y_j = f_j(x) y_i,\, i,j=1,\ldots,s\}
$$
that projects surjectively onto $\C^m$ and $\sigma$ is the projection on the first factor.  Then $M$ is 
the union 
of finitely many standard affine charts $\mathcal V_i$, where 
$$
\mathcal V_i= M \cap  \{(x, y_1, \ldots ,\widehat{y_i} \ldots ,y_s) \in \C^m \times \C^{s-1}; 
f_i(x) y_j = f_j(x) ,\, j=1,\ldots,\widehat i, \ldots ,s\}
$$ is a subvariety of the affine space 
$\C^m \times \C^{s-1}= \{(x, y) \in \C^m \times \proj^{s-1}; y_i=1\}$ with the coordinates $x_j$, $j=1, \ldots ,m$, and 
$y_j= f_j/f_i$, $j=1, \ldots, \widehat i, \ldots , s$.  The pullback of $\sI$ on $\mathcal V_i$ is generated 
by $f_i$,  since $f_j =y_j  f_i$   for $j\ne i$,
and hence is an invertible ideal, i.e.  principal and generated by a non-zero divisor.  The pullback of $\sI$ on $M$, denoted by  $\sigma^*(\sI)$, is an invertible sheaf of ideals.  

In general, for arbitrary $\sI$, the blow-up space $Bl_\sI \C^m$ is singular.   
Using Hironaka's resolution of singularities \cite{Hironaka64},  
it is possible to give an algorithm of principalization of any ideal by composing a sequence of 
smooth blow-ups (that is blowing-ups 
with smooth nowhere dense centers).  Then the blow-up space is non-singular.   
Such a principalization map is not unique.  
  We will use the following theorem  that is a special case of  Principalization III of 
\cite{Kollar07}, see also Theorem 1.10 of \cite{BM97}. 

\begin{thm}\label{kollar3}
Let $\sI \ne (0)$ be an ideal of $\C [x]$. 
Then there exists a composition of smooth blowing-ups, $\sigma :  M\to \C^m$,  such that : 
\begin{enumerate}
\item
$\sigma^*(\sI)$ is invertible and its zero set $V(\sigma^*(\sI))$ is a simple normal crossing divisor.   
\item
The restriction $\sigma|_{M\setminus V(\sigma^*(\sI))} : M\setminus V(\sigma^*(\sI)) \to \C^m \setminus V(\sI)$ 
is an isomorphism. 
\end{enumerate} 
\end{thm}

Recall that a simple normal crossing divisor $E\subset M$ is the union $E = \cup_i E_i$ of 
nonsingular hypersurfaces $E_i$ intersecting transversally. 


\subsection{Privileged coordinate system}  \label{coordinates}

The map $\sigma$ of Theorem \ref{kollar3}  is itself a blowing-up of an ideal $\sK \subset \C [x]$ such that 
$V(\sK)  = V(\sI)$.  
(Because $\sigma$ is birational and projective,  see \cite{Hartshorne77}, Ch. II Thm. 7.17 and 
Exercise 7.11 (c)).  Thus $M=Bl_\sK \C^m$.  Denote  $V(\sigma^*(\sI))$ by $E$ and let 
 $E = \cup_i E_i$ be the decomposition into irreducible components.  

Let $\sI=(f_1, \ldots, f_s)$ and $\sK=(h_1, \ldots, h_l)$ and let  
$\mathcal V\subset Bl_\sK \C^m$ be a standard affine chart of $Bl_\sK \C^m$.  Then on $\mathcal V$, 
$\sigma^* (\sI)$ is a principal ideal that is generated by some $f_i $, say $f_1 $, and  
$\sigma^* (\sK)$  by some $h_i $, say $h_1$.      
$\mathcal V$ is a subvariety of 
$\C^m \times \C^{l-1}$ with the coordinates $x_i$, $i=1, \ldots ,m$, and 
$ h_j/h_{1}$, $j=2, \ldots , l$.    
Each component $E_i\cap \mathcal V$
 is the zero set of a finite number of functions regular on $\mathcal V$, that is functions of the form $P/h_1^s$, 
 $s\in \N$, where $P$ is a polynomial in $x$.

  Let $p\in \mathcal V$. There is a neighborhood $\mathcal U$ of $p$ in $\mathcal V$, and a coordinate system 
$y_1, \ldots, y_m$ on $\mathcal U$, such that $y_i = P_i /h_1^{s}$,  $P_i \in \C[x]$, $s\in \N$,  and 
$E\cap  \mathcal U$ is given by $y_1 \cdots y_r=0$.  (By taking the maximum we may choose $s$ independent 
of $i$.) Then 
\begin{align}\label{normalcrossings}
 f_1 = unit \cdot \prod_{i=1}^r y_i^{n_i},  \quad h_1 = unit \cdot \prod_{i=1}^r y_i^{m_i}, 
\end{align}
with $n_i>0$ and $m_i> 0$.  Here by a unit we mean a function defined, analytic, and  nowhere vanishing  on $\mathcal U$.


\subsection{Tower of smooth principalizations.}

For  the ideals $\sD_{\kk}$, $\kk =2, \ldots, n$, 
we construct a tower of smooth principalizations  
\begin{align}\label{tower}
M_1 = \C^n  \stackrel{ \sigma_2 }\longleftarrow M_2  \stackrel{ \sigma_{3,2} }\longleftarrow  M_3 
\stackrel{ \sigma_{4,3} }\longleftarrow  \cdots \stackrel{ \sigma_{n,n-1} }\longleftarrow M_{n} .
\end{align}
We take as  $\sigma_2$ a smooth principalization of $\sD_{2}$ satisfying the conclusion of  Theorem \ref{kollar3}. 
Given  $ \sigma _2,  \sigma _{3,2}, \ldots, \sigma_{\kk,\kk-1}$,  we denote by  $\sigma_\kk :M_\kk\to \C^n$ the composition $\sigma_\kk = \sigma _2 
\circ \sigma _{3,2} \circ \cdots \circ \sigma_{\kk,\kk-1}$,  and take as 
$ \sigma_{\kk+1,\kk}$  a smooth principalization of $\sigma_\kk^*(\sD_{\kk +1})$.   Then $ \sigma_{\kk+1} = \sigma_\kk\circ  \sigma_{\kk+1,\kk}$ is  a smooth principalization of $\sD_{\kk +1}$.  
We denote $\sigma_{n,\kk} = \sigma_{\kk+1,\kk} \circ \cdots \circ \sigma_{n,n-1}$. 

By Subsection \ref{coordinates}, $\sigma _\kk$ is the blowing-up of an ideal $\sK_\kk \subset \C[a]$.  

\begin{defi}\label{localdata}
By  \emph{local data} $(f, h,  P_i,  s, r )$ for $p\in M_\kk$ we mean the following.  A polynomial $f \in \sD_{\kk}$ that generates $\sigma_\kk^*(\sD_{\kk})$ at $p$,  a polynomial  $h \in \sK_{\kk}$ that generates $\sigma_\kk^*(\sK_{\kk})$,  
a positive integer $s$ 
and polynomials $P_i$ such that  $y_i = P_i /h^{s}$, $i=1, \ldots , n$, is a privileged system of coordinates 
in a neighborhood $\mathcal U$ of $p$, and $r$ such that $f^{-1} (0) $ is given by $y_{1}  \cdots y_{r}=0$.
\end{defi}

We  fix  such local data for every $p \in M_\kk$  (but allow to replace the neighborhood $\mathcal U$ by a smaller one if necessary).


\begin{defi}\label{chain}
By a \emph{chain $\sfC= (p_{\kk}, f_{\kk}, h_\kk, P_{\kk,i}, s_\kk , r_\kk)$ for $p_n\in M_n$} 
we mean the  points  $p_\kk : = \sigma _{n,\kk} (p_n) $, $\kk=1, \ldots, n$, and the local data 
$(f_{\kk},  h_\kk,  P_{\kk, i}, s_\kk, r_\kk)$ for $p_\kk$.  We complete this data for $\kk=1$ by putting 
 $f_1 =h_1=1$, $P_{1,i}=a_i$, and $s_1=r_1=0$.   

When we specify the neighborhoods $\mathcal U_{\kk} \subset M_\kk$ of $p_\kk$ on which these local data are defined    
we always assume that $\sigma _{\kk,\kk-1} (\mathcal U_{\kk})\subset \mathcal U_{\kk-1}$.     
\[
  \xymatrix@C=1.5cm{
  M_1 & \ar[l]_{\si_2} M_2 & \ar[l]_{\si_{3,2}} M_3 &  \ar[l]_{\si_{4,3}} \cdots & \ar[l]_{\si_{n-1,n-2}} M_{n-1} &  \ar[l]_{\si_{n,n-1}} M_n \\
\ar@{^{ (}->}[u] \mathcal  U_1 & \ar@{^{ (}->}[u] \ar[l]_{\si_2} \mathcal  U_2 & \ar@{^{ (}->}[u] \ar[l]_{\si_{3,2}} 
  \mathcal  U_3 &  \ar[l]_{\si_{4,3}} \cdots 
  & \ar@{^{ (}->}[u] \ar[l]_{\si_{n-1,n-2}} \mathcal  U_{n-1} & \ar@{^{ (}->}[u] \ar[l]_{\si_{n,n-1}} \mathcal  U_n \\
  \ar@{}[u]|-*[@]{\in} p_1 & \ar@{}[u]|-*[@]{\in} \ar@{|->}[l]_{\si_2} p_2 & \ar@{}[u]|-*[@]{\in} \ar@{|->}[l]_{\si_{3,2}} p_3 &  
  \ar@{|->}[l]_{\si_{4,3}} \cdots & \ar@{}[u]|-*[@]{\in} \ar@{|->}[l]_{\si_{n-1,n-2}} p_{n-1} & \ar@{}[u]|-*[@]{\in} \ar@{|->}[l]_{\si_{n,n-1}} p_n 
  }
\]
\end{defi}
\medskip

We pull back the polynomial $P_a$ onto 
$M_\kk$ via $\sigma_\kk$, 
$$
P_{\sigma_\kk^*(a)} (Z) = Z^n + \sum _{i=1}^n( a_i\circ \sigma_\kk) Z^{n-i}.
$$
The  roots of $P_{\sigma_n^*(a)}$ are the pull-backs of the roots of $P_a$.    


\subsection{Formulas for the roots.} \label{ssec:formulas}

\begin{thm}\label{roots}{\rm [Formulas for the roots] } \\
Given a tower \eqref{tower},  
we may associate with every $p_\kk \in M_\kk$ a convergent power series $\psi_\kk$, 
an integer $q_\kk\ge 1$, and  a positive exponent $\al_\kk \in  \frac 1 {q_\kk}\N_{>0}$, such that the following holds.  
For any chain  $\sfC= (p_{\kk}, f_{\kk}, h_\kk, P_{\kk,i}, s_\kk , r_\kk)$  the roots of  
 $P_{\sigma_n^*(a)} $ in a neighborhood of  $p_n$ are given by  
 \begin{align}\label{sums}
 \sum_{\kk=1}^n A_\kk \,  \, {\varphi_\kk  \circ \sigma_{n,\kk}} ,
\end{align} 
 where  $ A_\kk\in \Q$ and 
 \begin{align}\label{varphik}
\varphi_\kk & = f_\kk^ {\al_\kk}  \psi_\kk (y_{\kk,1}^{1/q_\kk}, \ldots , y_{\kk, r_\kk}^{1/q_\kk}, y_{\kk,r_\kk+1}, \ldots , y_{\kk,n}) . 
\end{align}
\end{thm}

Theorem \ref{roots} will be proved in Section \ref{proofofmaintheorem}. 

\begin{rem}\label{ramification} 
Because $f_\kk$ is a normal crossing in $y_{\kk,i}$, cf.  \eqref{normalcrossings},  
$$f_\kk^{\al_\kk } \in \C\{(y_{\kk,1}^{1/q_\kk}, \ldots , y_{\kk, r_\kk}^{1/q_\kk}, y_{\kk,r_\kk+1}, \ldots , y_{\kk,n}) \}.$$
Hence $\varphi_\kk$ of   \eqref{varphik} is a fractional power series.  It can be  interpreted geometrically as follows.  Set 
\begin{align}\label{branching}
y_{\kk,i} = \begin{cases}
 t_{i}^{q_\kk} & \text {  if } i\le r_\kk  \\
 t_i  &  \text{ if }    i > r_\kk+1. 
\end{cases}
\end{align}
  Then  $\varphi_\kk$ is a convergent power series in 
$t=(t_1, \ldots , t_n)$.  There are neighborhoods $\mathcal U_\kk$  of $p_\kk$, 
$\sigma _{\kk,\kk-1} (\mathcal U_{\kk})\subset \mathcal U_{\kk-1}$, and their branched covers 
$\tau_\kk : \widetilde {\mathcal U}_\kk \to \mathcal U_\kk$, given by the formulas \eqref{branching}, such that 
$\psi_\kk\circ \tau_\kk$ and $\varphi_\kk\circ \tau_\kk$ are analytic on $\widetilde {\mathcal U}_\kk$.  
Since $\sigma_{\kk+1,\kk}^{-1}  (f_\kk^{-1}(0) )\subset f_{\kk+1} ^{-1} (0)$,  $y_{\kk,i}\circ  \sigma_{\kk+1,\kk}$, 
for $i\le r_\kk$,  is a normal crossings in $y_{\kk+1,1}, \ldots , y_{\kk+1, r_{\kk+1}}$ and therefore, we may 
suppose that $\sigma_{\kk+1,\kk} \circ \tau _{\kk+1}$ factors through $\tau_\kk$, changing $q_{\kk+1}$ if necessary.   Thus we obtain a sequence of branched covers $\tau_i$ making the following diagram commutative.
\[
  \xymatrix@C=1.5cm{
\widetilde   {\mathcal U}_1 & \ar[l]_{\widetilde \si_2} \widetilde  {\mathcal U}_2 &
 \ar[l]_{\widetilde \si_{3,2}} \widetilde  {\mathcal U}_3 
&  \ar[l]_{\widetilde \si_{4,3}} \cdots & \ar[l]_{\widetilde \si_{n-1,n-2}} \widetilde {\mathcal U}_{n-1} 
&  \ar[l]_{\widetilde \si_{n,n-1}} \widetilde {\mathcal U}_n \\
  \ar@{=}[u]  {\mathcal U}_1 & \ar@{<-}[u]_{\tau_2} \ar[l]_{\si_2}  {\mathcal U}_2 
  &  \ar@{<-}[u]_{\tau_3}  \ar[l]_{\si_{3,2}}  {\mathcal U}_3 &  \ar[l]_{\si_{4,3}} \cdots 
  & \ar@{<-}[u]_{\tau_{n-1}} \ar[l]_{\si_{n-1,n-2}}  {\mathcal U}_{n-1} & \ar@{<-}[u]_{\tau_n} \ar[l]_{\si_{n,n-1}} 
   {\mathcal U}_n  
  }
\]
Then Theorem  \ref{roots} says that the roots of  $P_{\tilde \si_n^*(a)} $ are  combinations of analytic functions on $\widetilde {\mathcal U}_n$ that are pull-backs of analytic functions on the $\widetilde {\mathcal U}_\kk$'s. 
\end{rem}

\begin{defi}\label{extendedchain}
By an \emph{extended chain} 
$\sfE= (p_{\kk}, f_{\kk}, h_\kk, P_{\kk,i}, s_\kk , r_\kk, \mathcal U_\kk)$ for $p_n\in M_n$ we mean a 
chain $\sfC= (p_{\kk}, f_{\kk}, h_\kk, P_{\kk,i}, s_\kk , r_\kk)$ and a system of neighborhoods  
  $\mathcal U_{\kk}$ of $p_\kk$ as in  Remark \ref{ramification}.  By Theorem \ref{roots} for every $p_n\in M_n$ there is an extended chain.  
\end{defi}

We filter the coefficient space $\C^n$ by the zero sets of discriminant ideals $\Sigma_\kk : = V(\sD_{\kk})$,
$$
\C^n \supset \Sigma_n \supset \cdots \supset \Sigma_2.  
$$
By Corollary \ref{zeroset}, $a\in \Sigma_\kk$ if and only if $P_a (Z)$ has at most $\kk-1$ distinct roots.  
If $a(t): \R \supset I \to \C^n$ is continuous 
then 
$\Omega_\kk := I \setminus a^ {-1} (\Sigma_\kk)$ defines a filtration by open subsets 
$
I \supset \Om_2 \supset \cdots \supset \Om_n$.  Because $\sigma_\kk $ is an isomorphism over 
$\C^ n\setminus \Sigma_\kk$, $a|_{\Om_\kk}$ has a lift $\lifta_\kk$ to $M_\kk$. For $\kk=n$ we write $\lifta := \lifta_n$.
\[	\xymatrix{
	  		&& M_n  \ar[d]^{\sigma_n} &&  \\
	  		I\supset \Om_n  \ar[rr]^{a} \ar[rru]^{\lifta}&& \C^n && 
      	}
\]

  \begin{lem}\label{addendum}{\rm [Addendum to Theorem \ref{roots}] } \\
Let $\sfE= (p_{\kk}, f_{\kk}, h_\kk, P_{\kk,i}, s_\kk , r_\kk,\mathcal U_\kk)$ be an extended  chain for $p_n\in M_n$  and let  $J$ be a connected component of $ \lifta^{-1} (\mathcal U_{n}) $. 
Let $\lambda (t)$ be a continuous root of $P_{a(t)}(Z)$ on $J$.  Then there are continuous choices of radicals 
$f_\kk^ {\al_\kk} (a(t))$ and $y_{\kk,1}^{1/q_\kk} (\lifta_\kk(t)), \ldots , y_{\kk, r_\kk}^{1/q_\kk}
 (\lifta_\kk(t))$, 
  such that 
 \begin{align}\label{sumst}
\lambda(t) =  \sum_{\kk=1}^n A_\kk \,  \, \varphi_\kk (t),
\end{align} 
 where 
 \begin{align}\label{varphikt} 
\varphi_\kk (t) & = f_\kk^ {\al_\kk}  (a(t))  \psi_\kk (y_{\kk,1}^{1/q_\kk}  (\lifta_\kk(t)), \ldots , y_{\kk, r_\kk}^{1/q_\kk}  (\lifta_\kk(t)) , y_{\kk,r_\kk+1} (\lifta_\kk(t)), \ldots , y_{\kk,n} (\lifta_\kk(t)) ). 
\end{align}
\end{lem} 

Lemma \ref{addendum} will be proved in Subsection \ref{proofaddendum}.


\section{Function spaces} \label{sec:functspaces}

In this section we fix notation for function spaces and present an extension lemma. 

\subsection{Function spaces}
Let $\Om \subset \R^n$ be open and bounded. We denote by $C^0(\Om)$ the space of continuous complex-valued functions on $\Om$.
For $k \in \N \cup \{\infty\}$ we set 
\begin{align*}
  C^k(\Om) &= \{f \in \C^\Om ; \p^\al f \in C^0(\Om), 0 \le |\al| \le k\},\\
  C^k(\overline \Om) &= \{f \in C^k(\Om) ; \p^\al f \text{ has a continuous extension to } \overline \Om, 0 \le |\al| \le k\},
\end{align*}
where $\overline \Om$ denotes the closure of $\Om$.

Note that $C^{k}(\overline \Om)$ is a Banach space when equipped with the norm
\[
\|f\|_{C^{k}(\overline \Om)} 
:= \sup_{\substack{|\al| \le k\\ x \in \Om}} |\p^\al f(x)|.
\]

For $k \in \N$ and $p \ge 1$ we consider the Sobolev space
\[
	W^{k,p}(\Om) = \{f \in L^p(\Om) ; \p^\al f \in L^p(\Om), 0 \le |\al| \le k\},
\]
where $\p^\al f$ denote distributional derivatives. On bounded intervals $I \subset \R$ the Sobolev space $W^{1,1}(I)$ 
coincides with the space $AC(I)$ of absolutely continuous functions on $I$
if we identify each $W^{1,1}$-function with its unique continuous representative. 
Recall that a function $f : \Om \to \R$ on an open subset $\Om \subset \R$ 
is absolutely continuous if for every $\ep>0$ there exists $\de>0$ so that  
$\sum_{i=1}^n |a_i -b_i| < \de$ implies $\sum_{i=1}^n |f(a_i) -f(b_i)| < \ep$ whenever $[a_i,b_i]$, $i =1,\ldots,n$, 
are non-overlapping intervals contained in $\Om$.
 
Let $\Om \subset \R^n$ be open and bounded, and let $1 \le p < \infty$. 
A measurable function $f : \Om \to \C$ belongs to the weak $L^p$-space $L_w^p(\Om)$ if 
\[
\|f\|_{p,w,\Om} := \sup_{r\ge 0} \Big\{r \cdot \cL^n(\{x \in \Om ; |f(x)| > r\})^{\frac{1}{p}}\Big\} < \infty,
\]
where $\cL^n$ denotes the $n$-dimensional Lebesgue measure. 
For $1 \le q < p < \infty$ we have (cf.\ \cite{Grafakos08} Example 1.1.11)
\begin{equation} \label{eq:qp}
  \|f\|_{q,w,\Om} \le \|f\|_{L^q(\Om)} \le \Big(\frac{p}{p-q}\Big)^{\frac{1}{q}} 
  \cL^n(\Om)^{\frac{1}{q}-\frac{1}{p}} \|f\|_{p,w,\Om}
\end{equation}
and hence
$L^p(\Om) \subset L_w^p(\Om) \subset L^q(\Om) \subset L_w^q(\Om)$
with strict inclusions. 
Note that $\|~\|_{p,w,\Om}$ is only a quasinorm; more precisely, for $f_j \in L_w^p(\Om)$ we have
\begin{equation} \label{triangle}
	\Big\|\sum_{j=1}^m f_j \Big\|_{p,w,\Om} \le m \sum_{j=1}^m \|f_j\|_{p,w,\Om}.	
\end{equation}
If $\Om_i$ is a finite or countable family of open sets whose union is $\Om$ then (cf.\ \cite{GhisiGobbino13} Lemma 3.1)
\begin{align} \label{subadditivity}
    &\| f \|_{p,w, \Om} \le \sum_i  \|f \|_{p,w, \Om_{i}} , \qquad \forall f\in L^p_w (\Om). 
\end{align}
If $p>1$ then there exists a norm equivalent to $\|~\|_{p,w,\Om}$ which makes $L_w^p(\Om)$ into a Banach space.

Analogously we may consider $L^p_w(K)$ for compact sets $K \subset \R^n$.  
We shall also use $W^{k,p}_{\on{loc}}$, $AC_{\on{loc}}$, as well as $L^p(\Om,\C^m) = [L^p(\Om)]^m$, $W^{k,p}(\Om,\C^m) = [W^{k,p}(\Om)]^m$, etc.,\ with the obvious meaning. 

\subsection{Extension lemma} 
The following lemma is a generalization of Lemma 3.2 of \cite{GhisiGobbino13} with essentially the same proof. 

\begin{lem} \label{extend}
  Let $\Om \subset \R$ be open and bounded, let $f : \Om \to \C$ be continuous, and set $\Om_0 := \{t \in \Om ; f(t) \ne 0\}$.  
  Assume that $f|_{\Om_0} \in AC_{\on{loc}}(\Om_0)$ and that $f|_{\Om_0}' \in L_w^p(\Om_0)$ for some $p>1$ 
  (note that $f$ is differentiable a.e.\ in $\Om_0$).
  Then the distributional derivative of $f$ in $\Om$ is a measurable function $f' \in L_w^p(\Om)$ and 
  \begin{equation} \label{eq:extend}
  \|f'\|_{p,w,\Om} = \|f|_{\Om_0}'\|_{p,w,\Om_0}.
  \end{equation}
\end{lem}

\begin{proof} 
The function $\ps : \Om \to \C$ defined by  
\[
\ps(t) := 
\begin{cases}
  f'(t) & \text{ if } t \in \Om_0\\
  0     & \text{ if } t \in \Om \setminus \Om_0  
\end{cases}
\]
clearly belongs to $L_w^p(\Om)$. We show that $\ps$ is the distributional derivative of $f$ in $\Om$.
Let $\ph \in C^\infty_c(\Om)$ be a test function with compact support in $\Om$ and 
let $\cC$ denote the (at most countable) set of connected components of $\Om_0$. 
Then, using integration by parts for the Lebesgue integral (see e.g.\ \cite{Leoni09} Corollary 3.37) 
\begin{align*}
  \int_\Om f \ph' \, dt = \int_{\Om_0} f \ph'  \, dt = \sum_{J \in \cC} \int_{J} f \ph'  \, dt
  = -\sum_{J \in \cC} \int_{J} f' \ph  \, dt = -\int_{\Om_0} f' \ph  \, dt = -\int_{\Om} \ps \ph  \, dt.   
\end{align*}  
(If $J=(a,b)$ then $\int_a^b f \ph'  \, dt = \lim_{\ep \to 0^+} \int_{a+\ep}^{b-\ep} f \ph'  \, dt 
= - \lim_{\ep \to 0^+} \int_{a+\ep}^{b-\ep} f' \ph  \, dt
= - \int_a^b f' \ph  \, dt$, by the Dominated Convergence Theorem, continuity of $f$, and \eqref{eq:qp}.)
Moreover, we have $\|f'\|_{p,w,\Om} = \|\ps\|_{p,w,\Om} = \|\ps\|_{p,w,\Om_0} = \|f|_{\Om_0}'\|_{p,w,\Om_0}$.
\end{proof}


\bigskip

\section{Absolute continuity of  roots} \label{absolutecontinuity} 

\subsection{Optimal regularity of radicals of differentiable functions}

We need the following variant of Theorem 2.2 of \cite{GhisiGobbino13}.

\begin{prop} \label{prop:radicals}
  Let $I \subset \R$ be a bounded interval and let $k \in \N_{>0}$.
  For each $f \in C^{k}(\overline I,\C)$ we have 
  \begin{equation} \label{est1}
    |f'(t)| \le \La_{k}(t) |f(t)|^{1-\frac{1}{k}} \quad \text{ a.e.\ in } I
  \end{equation}  
  for some $\La_{k}=\La_{k,f} \in L_w^p(I,\R_{\ge0})$, where $\frac{1}{p} + \frac{1}{k} = 1$, and such that
  \begin{equation} \label{est2}
    \|\La_{k}\|_{p,w,I} \le C(k) \max\big\{  \|f^{(k)}\|_{L^\infty(I)}^{\frac{1}{k}}  |I|^{\frac{1}{p}}, 
    \|f'\|_{L^\infty(I)}^{\frac{1}{k}}\big\}.
  \end{equation}
\end{prop}

\begin{proof}
  If the real and imaginary part of $f$ satisfy \eqref{est1}, then so does $f$. Hence it suffices to 
  consider the case that $f$ is real-valued.
  
Set $h = |f|^{\frac{1}{k }}$.   
  Then Theorem 2.2 of \cite{GhisiGobbino13} implies that $h' \in L_w^p(I,\R)$ and 
  \begin{equation*}
    \|h'\|_{p,w,I} \le C(k) \max\big\{  \|f^{(k)}\|_{L^\infty(I)}^{\frac{1}{k}}  |I|^{\frac{1}{p}}, 
    \|f'\|_{L^\infty(I)}^{\frac{1}{k}}\big\}.
  \end{equation*}
  In particular, $h \in W^{1,q}(I,\R)$ for each $q \in [1,p)$. 
  By differentiating $h^{k} = |f|$ we find 
  \[
  |f'(t)| = k  |h'(t)||f(t)|^{1-\frac{1}{k}}\quad  \text{ for all } t\in I \text{ with } f(t)\ne0.
  \]  
  The derivative $f'$ vanishes at the accumulation points of $f^{-1}(0)$, and 
  the isolated points of $f^{-1}(0)$ form an at most countable set. 
  So we conclude that 
  \eqref{est1} holds with $\La_{k} := k |h'|$.
\end{proof}

\begin{rem}
  Proposition~\ref{prop:radicals} is optimal in the following sense.   $\La_{k}$ cannot, in general, be chosen in $L^p$. Indeed, for $f : (-1,1) \to \R$, $f(t)=t$, we have 
    \[
    \Big(\frac{|f'|}{|f|^{1-\frac{1}{k}}}\Big)^p = \big(|t|^{\frac{1}{k}-1}\big)^p = |t|^{-1},
    \] 
    which is not integrable near $0$. See \cite{GhisiGobbino13} Example 4.3.
\end{rem}

\begin{rem}
In Proposition \ref{prop:radicals} it is actually enough to require that 
$f \in C^{k-1,1}(\overline I,\C)$; cf.\ \cite{GhisiGobbino13} Theorem 2.2.
\end{rem}


\subsection{Set-valued functions and curves of polynomials}
In the following we shall be dealing with multi-valued functions arising from complex radicals, their composition with single-valued 
functions, and their addition and multiplication. 

The (usual) composition $G \o F : X \leadsto Z$ of two set-valued functions $F : X \leadsto Y$ and $G : Y \leadsto Z$ is given by 
$(G \o F)(x) = \cup_{y \in F(x)} G(y)$. 
The addition $F+G$ and multiplication $F G$ of $F : X \leadsto \C$ and $G : X \leadsto \C$ are then well-defined.

A \emph{selection} of a set-valued function $F : X \leadsto Y$ is a single-valued function $f : X \to Y$ satisfying 
$f(x) \in F(x)$ for all $x \in X$. 
A \emph{parameterization} of a set-valued function $F : X \leadsto Y$ is a pair $(f,Z)$, where
$f : X \times Z \to Y$ is a single-valued function so that $F(x) = \{f(x,z) ; z \in Z\}$
for every $x \in X$. 
We shall only be concerned with multi-valued functions $F$ so that the 
cardinality $|F(x)|$ is finite and bounded, i.e., $\max_{x\in X} |F(x)| =:N < \infty$.
Then a parameterization of $F$ is an $N$-tuple of single-valued functions
(with multiplicities at points $x$ where $|F(x)|<N$).

If the coefficients of the polynomial $P_a$ in \eqref{polynomial} are complex-valued continuous functions $a_j \in C^0(I)$ defined in an 
interval $I \subset \R$, we say that $P_a(t) = P_{a(t)}$, $t \in I$, is a curve of polynomials. 
The roots of a curve of polynomials form a multi-valued function $\la : I \leadsto \C$ which admits a continuous parameterization; 
see \cite{Kato76} Chapter II Theorem~5.2.  
(This is no longer true if the parameter space is higher-dimensional due to 
monodromy.)  
Moreover, any continuous selection of $\la : I \leadsto \C$ can be completed to a continuous parameterization 
$\lambda_1, \ldots, \lambda_n$ such that $P_{a(t)} (Z) = \prod_i (Z-\lambda_i(t))$; 
see \cite{RainerN} Lemma 6.17.

\begin{lem} \label{estimate}
Let $\al \in \Q_{>0}$, $q, s \in \N_{>0}$, and suppose that 
$ k \ge \lceil \max \big\{\frac {s} {\al},{q}\big\} \rceil$,  $p = \frac{k}{k-1}$.  
Let $I \subset \R$ be a bounded interval and let $\mathcal U \subset \C^n$ be open and bounded.
Let $\ps \in C^1(\overline {\mathcal U})$,
$h, P_j \in C^{k}(\overline I)$, and let $\Om \subset I$ be an open subset of $I$  
so that $y^{1/q}(\Om) \subset \mathcal U$, where we put $y=(y_1,\ldots,y_n)=(P_1/h^s,\ldots,P_n/h^s)$ and 
$y^{1/q}=(y_1^{1/q},\ldots,y_n^{1/q})$.
Consider the multi-valued function
\[
\vh =h^\al \ps\Big(y^{\frac{1}{q}}\Big) = 
h^\al \ps\Big(y_1^{\frac{1}{q}},\ldots,y_n^{\frac{1}{q}}\Big).
\]
Then $\vh$ admits a continuous parameterization on $\Om$ and  for any such parameterization $\ph$  the distributional derivative of $\ph$ in $\Om$ is a measurable function $\ph' \in L^p_w(\Om)$ and 
\begin{align} \label{eq:estimate}
    &\|\ph'\|_{p,w,\Om} \le C_1(\al,s,q,\mathcal U) \|\ps\|_{C^1(\overline {\mathcal U})} N_{\alpha,k,s,I} (h) 
     \max_{j} \big\{H_{k,I}(h),H_{k,I}(P_j)
   \big\}
\end{align}
for a positive constant $C_1(\al,s,q, \mathcal U)$,  
where for any function $g \in C^{k}(\overline I)$ we set 
\[
N_{\alpha,k,s,I}(g) :=   \max \big\{   \|g\|_{L^\infty(I)}^{\al-\frac s k},    \|g\|_{L^\infty(I)}^{\al-\frac 1 k} \big\} , 
  \quad 	H_{k,I}(g) :=  \max\big\{  \|g^{(k)}\|_{L^\infty(I)}^{\frac{1}{k}}  |I|^{\frac{1}{p}}, 
    \|g'\|_{L^\infty(I)}^{\frac{1}{k}}\big\}. 
\] 
\end{lem}

\begin{proof}
First we show that $\vh$ admits a continuous parameterization on $\Om$. 
Consider the open subsets $\Om_1 \subset \Om_0 \subset \Om \subset I$ given by 
\[
  \Om_0 = \{t \in \Om ; h (t) \ne 0\} \quad \text{ and } \quad \Om_1 = \{t \in \Om_0 ; \forall j ~P_j \ne 0\}.
\]
Each multi-valued function $y_j^{1/q}=(P_j/h^s)^{1/q}$ has a continuous parameterization on $\Om_0$ 
and thus so does $\ps(y^{1/q})$. The multi-valued function $h^\al$ has a continuous parameterization on $I$, 
which vanishes on the zero set of $h$. Since $\ps(y^{1/q})$ is bounded on $\Om$, we may conclude that $\vh$ 
admits a continuous parameterization on $\Om$.  

Let $\ph$ be any continuous parameterization of $\vh$ on $\Om$. Abusing notation we denote by $\ph$ also any 
single component in the parameterization $\ph$. 
Then $\ph$ is $C^1$ on $\Om_1$ and its 
derivative satisfies
\begin{align} \label{eq:derph}
\begin{split}
	|\ph'| &\le \al   \Big|\ps\Big(y^{\frac{1}{q}}\Big) \Big|  \frac{|h'|}{|h|^{1-\al}}
	 + 
	\frac{1}{q} \sum_{j=1}^n \Big|\p_j \ps\Big(y^{\frac{1}{q}}\Big)\Big| 
	\bigg(\frac{|P_j'|}{|P_j|^{1-\frac{1}{q}}}  |h|^{\al-\frac{s}{q}} +s\Big|\frac{P_j}{h^s}\Big|^{\frac{1}{q}} 
	\frac{|h'|}{|h|^{1-\al}} \bigg)\\
	& \le \al \Big|\ps\Big(y^{\frac{1}{q}}\Big) \Big|   \frac{|h'|}{|h|^{1-\frac 1 k}}  |h|^{ \al -\frac{1}{k}} 
	  \\
	 &\quad +
	\frac{1}{q} \sum_{j=1}^n \Big|\p_j \ps\Big(y^{\frac{1}{q}}\Big)\Big| 
	\bigg(\frac{|P_j'|}{|P_j|^{1-\frac{1}{k}}}  \Big|\frac{P_j}{h^s}\Big|^{\frac{1}{q} - \frac 1 k} 
	  |h^s|^{\frac \al s-\frac{1}{k}} + s\Big|\frac{P_j}{h^s}\Big|^{\frac{1}{q}} 
	\frac{|h'|}{|h|^{1-\frac 1 k}} |h|^{ \al -\frac{1}{k}}  \bigg)
	\end{split}	
\end{align}

Next we claim that $\ph$ is locally absolutely continuous on $\Om_0$. 
Indeed, 
every continuous parameterization of $h^\al$, respectively $P_j^{1/q}$, is $AC$ on $I$ by Proposition~\ref{prop:radicals}, 
and consequently every continuous parameterization of $y_j^{1/q} = (P_j/h^s)^{1/q}$ is $AC_{\on{loc}}$ on $\Om_0$; 
note that on each compact subinterval of $\Om_0$ any continuous parameterization of $1/h^{s/q}$ is $C^1$. 
Since $\ps$ is $C^1$, we may infer from Lemma 2.1 of \cite{MarcusMizel72}
that each continuous parameterization of $\ps(y^{1/q})$, and thus of $h^\al \ps(y^{1/q})$, is   
locally absolutely continuous on $\Om_0$. This shows the claim.

In particular, $\ph$ is differentiable almost everywhere in $\Om_0$. 
We argue that \eqref{eq:derph} holds almost everywhere in $\Om_0$, if we define 
\[
	\frac{|P_j'|}{|P_j|^{1-\frac{1}{k}}} := 0 \quad \text{ on accumulation points of $P_j^{-1}(0)$.}
\]
Indeed, by Lemma 2.1 of \cite{MarcusMizel72} the chain rule holds almost everywhere 
 and the derivative of any continuous, and hence absolutely continuous, parameterization of $P_j^{1/q}$ 
exists almost everywhere and vanishes on accumulation points of $P_j^{-1}(0)$.
The isolated points of $P_j^{-1}(0)$ form an at most countable set.  

Applying Proposition~\ref{prop:radicals} we may conclude that
\begin{equation} \label{eq:derph2}
	|\ph'| \le \La_{k} \Ps \quad \text{ a.e. in } \Om_0 \text{ for some } \Ps \in L^\infty(\Om,\R)
	\text{ and } \La_k \in L_w^p(I,\R).
\end{equation}
Here we set $\La_k = \max \{\La_{{k},h},\La_{{k},P_j}\}$.

Extending $\Ps$ by $1$ on $I \setminus \Om$ and using $L^p_w \cdot L^\infty \subset L^p_w$ we obtain
\begin{align*} 
    |\ph'| \le \tilde \La_{k} \quad \text{ a.e. in } \Om_0 \text{ for some } \tilde \La_k \in L_w^p(I,\R).      
\end{align*}
Using Lemma~\ref{extend} we may conclude that
the distributional derivative of $\ph$ in $\Om$ is a measurable function $\ph' \in L^p_w(\Om)$.  

The estimate \eqref{eq:estimate} follows from \eqref{eq:derph2}, \eqref{est2}, and \eqref{eq:extend}.  
\end{proof}


\subsection{Main Theorem}

\begin{thm} \label{theorem}
For every $n\in \N_{>0}$ there is $k=k(n) \in \N_{>0}$ and  $p=p(n)>1$ such that the following holds.  
	Let $I\subset \R$ be a compact interval and let 
	\[
		P_{a(t)}(Z) = Z^n + \sum_{j=1}^n a_j(t) Z^{n-j} \in C^{k}( I)[Z] 	
	\] 
	be a monic polynomial with coefficients $a_j \in C^{k}( I)$, $j = 1,\ldots,n$.   
	\begin{enumerate}
		\item 
		Let $\la_j \in C^0(I)$, $j = 1,\ldots,n$, be a continuous parameterization of the roots of $P_a$ on $I$.
		Then the distributional derivative of each $\la_j$ in $I$ is a measurable function $\la_j' \in L^q(I)$ for every 
		$q \in [1,p)$. 
		In particular, each $\la_j \in W^{1,q}(I)$ for every $q \in[1,p)$.
		\item This regularity of the roots is uniform.  Let $\{P_{a_\nu} ; \nu \in \cN\}$,
		\[
			 P_{a_\nu(t)}(Z) = Z^n + \sum_{j=1}^n a_{\nu,j}(t) Z^{n-j} \in C^{k}( I)[Z], ~\nu \in \cN, 	
		\] 
		be a family of curves of polynomials, indexed by $\nu$ in some set $\cN$, so that the set of coefficients 
		$\{a_{\nu,j} ; \nu \in \cN, j=1,\ldots,n\}$ is bounded in $C^{k}( I)$. 
		Then the set
		\[
			\qquad \{\la_{\nu} \in C^0(I);  P_{a_\nu}(\la_\nu)=0
			\text{ on $I$}, ~\nu \in \cN\}
		\]			
		is bounded in $W^{1,q}(I)$ for every $q \in [1,p)$.	
	\end{enumerate}
\end{thm}

The rest of this section will be devoted to the proof of Theorem \ref{theorem}.  


\subsection{Definition of $k(n)$ and $p(n)$.} \label{ssec:kp}
Let $\sfE= (p_{\kk}, f_{\kk}, h_\kk, P_{\kk,i}, s_\kk , r_\kk, \mathcal U_\kk)$ be an extended chain.   
By \eqref{normalcrossings}, we may express $\varphi_\kk $ of \eqref{varphik} as follows 
 \begin{align}\label{varphiksecond}
\varphi_\kk = h_\kk^ {\tal_\kk} \tilde \psi_\kk (y_{\kk,1}^{1/\tilde q_\kk}, \ldots , y_{\kk,r_\kk}^{1/\tilde q_\kk}, 
y_{\kk,r_\kk+1}, \ldots , y_{\kk,n}), 
\end{align}
 where $\tal_k \in \frac 1 {\tilde q_\kk} \N_{>0}$ and $\tilde q_\kk$ is a positive integer possibly much 
 bigger than $q_\kk$.   
Then we define  
\[
k_\sfE := \max _\kk \, \lceil \max \Big\{\frac {s_\kk} {\tal_\kk},{\tilde q_\kk}\Big\} \rceil, 
\]

 We fix an open bounded neighborhood $\mathcal B$ of the origin in $\C^n$ and a finite family of  extended chains  
\begin{align}\label{eqcover}
\cover = \{ \sfE_j \} =\{ (p_{j, \kk},f_{j, \kk}, h_{j,\kk}, P_{j,\kk,i}, s_{j,\kk} , r_{j,\kk},\mathcal U_{j, \kk}) \}
\end{align}
such that 
\begin{align*}
\sigma_n ^{-1} (\overline {\mathcal B}) \subset \bigcup_j \mathcal U_{j,n}.  
\end{align*}
Then we set    
\begin{align*}
	k = k(n)  := \max_{j}  k_{\sfE_j}, 	 \quad \text{ and } \quad
	p = p(n)  := \frac{k}{k-1} \, \in \Q_{>1} \cup\{\infty\}.
\end{align*}


\subsection{Real analytic case.} \label{analyticsubsection}
 We begin the proof of Theorem \ref{theorem} with the following special case. 
We suppose that $a(t)$ is real analytic and that 
$a(t) \in \mathcal B$ for all $t\in I$.  
We suppose moreover that the discriminant of $P_{a(t)}$ is not identically equal to zero. 
 Under these assumptions we show that  
\begin{align} \label{eq:estimate2}
    &\|\la_j'\|_{p,w, I} \le C(\cover) \sum _{j,\kk} N_{\tal_{j,\kk} ,k,s_{j,\kk} ,I}(h_{j,\kk} (t))
     \max_i \big\{H_{k ,I}(h_{j,\kk} (t)),H_{k ,I}(P_{j,\kk,i} (t))\big\},
\end{align}
where  the constant  $ C(\cover)$ depends only on the family $\cover$.  

Recall that $\lifta$ denotes the lift of $a|_{\Om_n}$ over $\si_n$, cf.\ Subsection \ref{ssec:formulas}. 
We remark that actually $a$ has a unique real analytic lift to $M_n$ on the whole interval $I$,
by the universal property of blowing-ups,  see \cite{Hartshorne77} Proposition 7.14; but we will not use this fact.

All the roots of $P_{a(t)}$ on $\Om_{n} $ are distinct and hence, by the Implicit Function Theorem,  depend analytically on $t$. 
Thus Lemmas \ref{estimate} and \ref{addendum} 
give \eqref{eq:estimate2} with $I$ replaced by $ \lifta^{-1} (\mathcal U_{n}) \cap \Om_n $.  
We set $I_i  := \lifta^{-1} (\mathcal U_{i,n})$ and $\Om_{i,n}:= I_{i} \cap \Om_{n}$. Then, 
by \eqref{subadditivity},  
\begin{align*} 
    &\|\la_j'\|_{p,w, \Om_n} \le \sum_i  \|\la_j'\|_{p,w, \Om_{i,n}}.
\end{align*}
Since $a(t)$ is real analytic,  $a^{-1}(\Sigma_n) = I \setminus \Om_{n} $ is finite, and hence 
the derivative $\la_j'$ of $\la_j$ exists almost everywhere in $I$ and belongs to $L^1(I)$ by \eqref{eq:qp}. 
It coincides with the distributional derivative of $\la_j$ in $I$, and 
$\|\la_j'\|_{p,w, I} = \|\la_j'\|_{p,w, \Om_n}$. 
(If $\ph \in C^\infty_c(I)$ and $(a,b)$ is a connected component of $\Om_n$, then 
$\int_a^b \la_j \ph' \,dt = \lim_{\ep \to 0^+} \int_{a+\ep}^{b-\ep} \la_j \ph' \,dt 
= \la_j(b) \ph(b) - \la_j(a) \ph(a) - \lim_{\ep \to 0^+} \int_{a+\ep}^{b-\ep} \la_j' \ph \,dt 
= \la_j(b) \ph(b) - \la_j(a) \ph(a) - \int_{a}^{b} \la_j' \ph \,dt$, by the Dominated Convergence Theorem. 
Since $I \setminus \Om_{n}$ is finite, all boundary terms cancel and thus $\int_I \la_j \ph' \,dt = - \int_I \la_j' \ph \,dt$.)
This implies \eqref{eq:estimate2}.


\subsection{Weighted homogeneity.}   
Let $a(t)$ be real analytic and  suppose that the discriminant of $P_{a(t)}$ is not identically equal to zero.    
We do not assume any longer that 
$a(t) \in \mathcal B$ for all $t\in I$.  We extend the bound of the previous subsection to such 
curves using the weighted homogeneity. 

 For $\eta>0$ and $a\in \C^n$ we define $\eta * a\in \C^n$ by $(\eta * a)_i = \eta^i a_i$.  Then $\lambda$ 
 is a root of $P_{a}$ if and only if $\eta  \lambda$ is a root of $P_{\eta *  a}$.   
 
 Fix $\rho \ge \max\{1,  \sup _{t\in I} \|a(t)\| \}$.
   Then $ \| \rho ^{-1}* a\| \le 1$.   For a polynomial $g\in \C[a]$, 
 set $\tilde g (t) : = g (\rho ^{-1}*  a (t))$.  Then  by \eqref{eq:estimate2}
\begin{align} \label{eq:estimate3}
    &\|\la_j'\|_{p,w, I} \le \rho\,  C(\cover) \sum _{j,\kk} N_{\tal_{j,\kk} ,k,s_{j,\kk} ,I}(\tilde h_{j,\kk} )
     \max_i \big\{H_{k ,I}(\tilde h_{j,\kk}) ,H_{k ,I}(\tilde P_{j,\kk,i} )\big\},
\end{align}



\subsection{General Case}  
Let $a\in C^{k}(I,\C^n)$.  By  the classical Weierstrass Theorem 
there is a sequence of polynomial curves $(a_\nu) \subset C^\om(I,\C^n)$, 
  such that 
  \[
  a_\nu \longrightarrow a \quad \text{ in } C^{k}(I,\C^n)  \quad (\nu \to \infty). 
  \]
    By replacing $a_\nu$ by $a_\nu + (0,\ldots,0,\varepsilon _\nu)$, with $\varepsilon_\nu > 0$ sufficiently small,   
    we may suppose moreover that  the discriminant of each $P_{a_\nu}$ is not identically zero.  
  
  For each $\nu$ choose a continuous parameterization 
  $\la_\nu = (\la_{\nu,1},\ldots,\la_{\nu,n}) \in C^0(I,\C^n)$ of the roots of $P_{a_\nu}(t)$, $t \in I$.  
  Since $(a_\nu)$ is bounded in $C^{k}(I,\C^n)$, we may infer from \eqref{eq:estimate2} that the set of distributional 
  derivatives $\{\la_\nu' ; \nu\}$ is bounded in $L^{q}(I,\C^n)$ for every $q \in [1,p)$. 
  
    Fix $q \in (1,p)$. 
  By the Arzel\'a--Ascoli Theorem, as $(\la_\nu)$ is equi-H\"older,    or alternatively by the   
  Rellich--Kondrachov Compactness Theorem, there is a subsequence $(\la_{\nu(\ell)})$ that 
  converges in $C^0(I,\C^n)$ to some $\la$. 

  Since $L^q(I)$ is reflexive, we also have (after possibly passing to a 
  subsequence again) that $(\la_{\nu(\ell)}')$ converges to some $\la'$ weakly in $L^q(I,\C^n)$. Then 
  $\la'$ is the distributional derivative of $\la$ and thus $\la \in W^{1,q}(I,\C^n)$.  
  It is clear that $\la$ forms a parameterization of the roots of $P_a$ on $I$.

\begin{lem}\label{comparison}
	Let $\la = (\la_1,\ldots,\la_n)$ and $\mu= (\mu_1,\ldots,\mu_n)$ be two parameterizations of the roots of 
	$P_a(t)(Z) \in C^{k}(I)[Z]$. If $\la \in W^{1,q}(I,\C^n)$, $q \in [1,p)$, and $\mu \in C^0(I,\C^n)$, 
	then also $\mu \in W^{1,q}(I,\C^n)$ and 
	\begin{equation} \label{eq2}
		 \sum_{i=1}^n \|\mu_i'\|_{L^q(I)}^q = \sum_{i=1}^n \|\la_i'\|_{L^q(I)}^q.  
	\end{equation}
\end{lem}

\begin{proof}
	For each $j$ we have 
	\[
		\on{length}(\mu_j) \le \sum_{i=1}^n \on{length}(\la_i) < \infty 
	\]
	and so $\mu_j$ is of bounded variation. Moreover, for any subset $E \subset I$  
	\[
		\mu_j(E) \subset \bigcup_{i=1}^n \la_i(E)
	\]
	and hence $\mu_j$ has the Luzin \thetag{N} property. We may conclude that each $\mu_j$ is absolutely continuous 
	on $I$ and hence the derivative $\mu_j'$ of $\mu_j$ exists almost everywhere in $I$ and coincides with the 
	distributional derivative of $\mu_j$ in $I$. 

	At points $t$, where each $\mu_j$ and each $\la_i$ is differentiable, the sets 
$\{\mu_j'(t)\}$ and $\{\la_i'(t)\}$ coincide together with the multiplicities of its elements. 	
	These points form a subset of $I$ of full measure and therefore $\mu_j' \in L^q(I)$ and satisfies \eqref{eq2}. 
\end{proof}

Uniformity can be seen by repeating the proof with an additional parameter $\nu$.
The weak limits in the reasoning above are weakly bounded and thus bounded in $L^q(I)$.

\newpage

\section{Multiparameter families of polynomials} \label{sec:mult}

\begin{thm} \label{multpar}  
Let $k=k(n) \in \N_{>0}$ and  $p=p(n)>1$ be as in Subsection \ref{ssec:kp}.  
Let $U \subset \R^m$ be open and let 
	\[
		P_{a(x)}(Z) = Z^n + \sum_{j=1}^n a_j(x) Z^{n-j} \in C^{k}(U)[Z] 	
	\] 
	be a monic polynomial with coefficients $a_j \in C^{k}(U)$, $j = 1,\ldots,n$.   
	\begin{enumerate}
		\item 
		Let $\la \in C^0(V)$ represent a root of $P_a$, i.e., $P_a(\la) = 0$, on a relatively compact open subset 
		$V \Subset U$. 
		Then the distributional gradient of $\la$ in $V$ is a measurable function $\nabla \la \in [L^q(V)]^m$ for every 
		$q \in [1,p)$.
		In particular, $\la \in W^{1,q}(V)$ for every $q \in[1,p)$.
		\item The regularity of the roots is uniform. Let $\{P_{a_\nu} ; \nu \in \cN\}$,
		\[
			 P_{a_\nu(x)}(Z) = Z^n + \sum_{j=1}^n a_{\nu,j}(x) Z^{n-j} \in C^{k}(U)[Z], ~\nu \in \cN, 	
		\] 
		be a family of polynomials, indexed by $\nu$ in some set $\cN$, so that the set of coefficients 
		$\{a_{\nu,j} ; \nu \in \cN, j=1,\ldots,n\}$ is bounded in $C^{k}(U)$. 
		Let $V \Subset U$. 
		Then the set
		\[
			\qquad \{\la_{\nu} \in C^0(V);  P_{a_\nu}(\la_\nu)=0
			\text{ on $V$}, ~\nu \in \cN\}
		\]			
		is bounded in $W^{1,q}(V)$ for every $q \in [1,p)$.	\end{enumerate}
\end{thm}

\begin{rem}
	The roots of a polynomial depending on at least two parameters do in general not admit a continuous parameterization due to monodromy.  For instance, the radical $\C \ni (x+iy) \mapsto (x+iy)^{1/n}$ 
	does not admit  continuous 
	parameterizations on $\C$. 
\end{rem}

\begin{proof}[Proof of Theorem \ref{multpar}]
	By Theorem~\ref{theorem}, $\la$ is absolutely continuous along each affine line parallel to 
	the coordinate axes. So $\la$ possesses the partial derivatives $\p_i \la$, $i=1,\ldots,m$, which 
	are defined almost everywhere and are measurable. 
	It clearly suffices to show that all partial derivatives $\p_j \la$ belong to $L^q(V)$, for every $q \in [1,p)$. 
	
	Set $x=(t,y)$, where $t=x_1$, $y=(x_2,\ldots,x_m)$, and let $V_1$ be the orthogonal projection of $V$ on the hyperplane $\{x_1=0\}$.
    For each $y \in V_1$ we denote by $V^y := \{t \in \R ; (t,y) \in V\}$ the corresponding section of $V$; note that $V^y$ is open in $\R$. 
    Then by Fubini's Theorem,
    \begin{equation} \label{Fubini}
    	\int_V |\p_1 \la(x)|^q\, dx = \int_{V_1} \int_{V^y} |\p_1 \la(t,y)|^q\, dt\, dy.
    \end{equation}
 	
 	We may cover $V$ by finitely many open boxes $K = I_1 \times \cdots \times I_m$ contained in $U$.
  Let $K$ be fixed and set $L = I_2 \times \cdots \times I_m$. 
  Fix $y \in V_1 \cap L$ and let $\la^y_j$, $j=1,\ldots,n$, be a continuous parameterization of the roots of $P_a(~,y)$ on 
  $\Om^y := V^y \cap I_1$ such that $\la(~,y) = \la^y_1$; it exists 
  since $\la(~,y)$ can be completed to a continuous parameterization of the roots of $P_a(~,y)$ on each connected component 
  of $\Om^y$
  by Lemma 6.17 of \cite{RainerN}. 
  Our goal is to bound  
  \[
  	\|\p_t\la(~,y)\|_{L^q(\Om^y)} = \|(\la^y_1)'\|_{L^q(\Om^y)}
  \]
  uniformly with respect to $y \in V_1 \cap L$. 

  To this end let $\cC^y$ denote the set of connected components $J$ of the open set $\Om^y$. 
  For each $J \in \cC^y$ extend the parameterization $\la^y_j|_J$, $j=1,\ldots,n$, continuously to $I_1$, i.e., choose
  \begin{gather*}
  	 	\la^{y,J}_j \in C^0(I_1), \quad j=1,\ldots,n, \quad \text{ so that } \\ 
  	 	\forall j\quad  \la^{y,J}_j|_J = \la^y_j|_J \quad   \text{ and } \quad 
  	 	P_a(t,y)(Z) = \prod_{j=1}^n (Z-\la^{y,J}_j(t)), ~t \in I_1.
  \end{gather*} 
  This is possible since $\la^y_j|_J$ has a continuous extension to the endpoints of the (bounded) interval $J$, by 
  Lemma 4.3 of \cite{KLMR05}, and can then be extended on the left and on the right of $J$ by a continuous parameterization 
  of the roots of $P_a(~,y)$ on $I_1$ after suitable permutations.

  By Theorem~\ref{theorem}, for each $y \in V_1\cap L$, each $J \in \cC^y$, and each $j=1,\ldots,n$, 
  $\la^{y,J}_j$ is absolutely continuous on $I_1$ and $(\la^{y,J}_j)' \in L^q(I_1)$ with 
  \begin{equation} \label{eq:ub}
  	\sup_{y,J,j} \|(\la^{y,J}_j)'\|_{L^q(I_1)} < \infty.
  \end{equation}
  Let $J,J_0 \in \cC^y$ be arbitrary. 
  By Lemma~\ref{comparison}, $(\la^y_j)'$ as well as $(\la^{y,J_0}_j)'$ belong to $L^q(J)$ and 
  we have  
  \begin{equation*}
  	\sum_{j=1}^n \|(\la^y_j)'\|_{L^q(J)}^q 
  	= \sum_{j=1}^n \|(\la^{y,J}_j)'\|_{L^q(J)}^q = \sum_{j=1}^n \|(\la^{y,J_0}_j)'\|_{L^q(J)}^q.
  \end{equation*}
  Thus, for arbitrary fixed $J_0 \in \cC^y$, 
  \begin{align*}
  	\sum_{j=1}^n \|(\la^y_j)'\|_{L^q(\Om^y)}^q 
  	&= \sum_{J \in \cC^y} \sum_{j=1}^n \|(\la^y_j)'\|_{L^q(J)}^q \\
  	&= \sum_{J \in \cC^y} \sum_{j=1}^n \|(\la^{y,J_0}_j)'\|_{L^q(J)}^q \\ 
  	&= \sum_{j=1}^n \|(\la^{y,J_0}_j)'\|_{L^q(\Om^y)}^q \\
  	&\le \sum_{j=1}^n \|(\la^{y,J_0}_j)'\|_{L^q(I_1)}^q
  \end{align*}
  In view of \eqref{eq:ub} we may conclude that   
  $\sup_{y \in V_1 \cap L} \|(\la^y_1)'\|_{L^q(\Om^y)} < \infty$. 
  By \eqref{Fubini} and since the number of boxes $K$ is finite, this implies that $\p_1 \la \in L^q(V)$.
  The other partial derivatives can be treated analogously. This shows (1).

  In order to see (2) it suffices to repeat the proof of (1) paying attention to the additional dependence on $\nu$.
\end{proof}


\part{Formulas for the roots.  Proof of Theorem \ref{roots}}\label{part2}

\bigskip
\section{Strategy of the proof} \label{strategy}
The main ideas of the proof of  Theorem \ref{roots} are the following.  
 Let $x=(x_1, \ldots , x_r)$ be local coordinates at $0\in \C^r$.  Suppose that $a_i \in \C\{x\}$ and 
let 
\[
		P_a(Z) = Z^n + \sum_{j=1}^n a_j Z^{n-j}.  	
	\] 
	Thus we may consider  $\sD_{\kk}$, defined in Section \ref{sec:formulas}, 
as an ideal of $\C\{x\}$.   If $a_1=0$ and $\sD_2$ is principal and generated by a monomial, then we may split 
$P_a$, that is factor it $P_a = P_b P_c$; see Step 2 of the proof, Section \ref{proofofmaintheorem}. 
This requires introducing fractional powers.  
If we can continue this process by splitting $P_b$, $P_c$, and then their factors, etc., then we eventually 
arrive at linear factors (i.e. of degree $1$) whose 
coefficients are the roots.  As we show in the next three sections 
this can be guaranteed by the principalization of the higher order discriminant ideals $\sD_{\kk}$.  
 
A subtle point of Theorem \ref{roots} and hence of its proof is to obtain  the exponent 
$\alpha_\kk$ 
 in \eqref{varphik} 
strictly positive.  This forces us to blow-up the ideals $\sD_\kk$ one by one.  Then we put each 
factor in Tschirnhausen  form, which amounts to subtracting a fraction of its first coefficient 
 from the roots.  The remaining part 
of the roots vanishes on $V(\sD_{\kk+1})$  and hence we may continue.  

This consecutive splitting process can be compared to the proof of the Abhyankar-Jung Theorem 
of \cite{ParusinskiRond2012}, that gives a formula for the roots in one shot by making the discriminant normal crossing, 
but without the property $\alpha_\kk>0$ which is crucial for us.  
In the splitting process of \cite{ParusinskiRond2012}, at each stage, the coefficients of the 
factors, say defined on $M_{\kk+1}$, are expressed in terms of their product, which is well-defined on $M_{\kk}$.  
The complexity of our proof comes from the fact that,  in the formula \eqref{varphik}, we need each $f_{\kk}$ to be 
the pull-back of a polynomial in the coefficients $a_i$, that is, of a polynomial defined on $\C^n$. 
Similarly, each $y_{\kk,i}$ has to be the pull-back of a rational function on $\C^n$.  

\section{A characterization of principality} \label{principality}

Let $x=(x_1, \ldots , x_r)$ be local coordinates at $0\in \C^r$.  Suppose that $a_i \in \C\{x\}$, 
for $i=1, \ldots , n$.  We denote by $\xi (x) = \{ \xi_1 (x), \ldots , \xi _n(x)\}$ the unordered set of roots of $P_a$.    
Let 
$$
\size_{\kk} (x) = \max_{|I|=\kk} \prod_{i\ne j \in I} |\xi_i (x) -\xi_j(x)|.   
$$
By definition  $\size_\kk$ is the germ at the origin of a non-negative real-valued function of $x\in \C^r$.

\begin{prop}\label{criterion}
\hfill
\begin{enumerate}
\item
There is a finite family  $g_1, \ldots, g_p \in \sD_{\kk}$ such that 
$  \size_\kk^{N!/\kk(\kk-1)} \sim \max_j |g_j|  .$
 Moreover, we may take as $g_1, \ldots, g_p$ any system of generators of $\sD_{\kk}$.  
 \item 
 $g\in \sD_{\kk}$ generates  $\sD_{\kk}$ if and only if $ \size_\kk^{N!/\kk(\kk-1)} 
 \sim |g| $.
\end{enumerate}
\end{prop}

\begin{proof}
\thetag{1}
We can choose as $g_1, \ldots, g_p $ the powers $\sigma_i ^{N!/i\kk(\kk-1)}$ of the elementary symmetric functions 
$\sigma_i $,  $i=1, \ldots , \binom  n \kk$,  in 
$\lambda_I=  \prod_{i\ne j \in I} (\xi_i (x) -\xi_j(x))$, $I\subset \{1, \ldots , n\}$, $|I|=\kk$.  
Indeed 
\begin{align*}
	\Big| &\sigma_s ( \lambda_I ; |I|=\kk ) \Big|^ {\frac{N!}{s\kk(\kk-1)}} 
	\le  \binom {\binom{n}{\kk}} s^ {\frac{N!}{s\kk(\kk-1)}}  \max _I |\lambda_I| ^ {\frac{N!}{\kk(\kk-1)}} 
	= \binom {\binom{n}{\kk}} s^ {\frac{N!}{s\kk(\kk-1)}} \size_\kk(x)^ {\frac{N!}{\kk(\kk-1)}}.
\end{align*}
The converse estimate follows from Lemma \ref{lem:boundedroots} below.

\thetag{2}
By (1) if $g$ generates  $\sD_{\kk}$ then $|g| \sim  \size_\kk^{N!/\kk(\kk-1)} $.  Suppose now that 
$|g| \sim  \size_\kk^{N!/\kk(\kk-1)} $. Then for any $f\in \sD_{\kk}$ the quotient $f/g$ is bounded and hence  
holomorphic.  
\end{proof}

\begin{lem}[\cite{Malgrange67} p.~56, or \cite{RS02} Theorem~1.1.4.] \label{lem:boundedroots}
	If $a_1,\ldots,a_n, z \in \C$ satisfy the equation $z^n + \sum_{j=1}^n a_j z^{n-j} =0$, 
	then $|z| \le 2 \max_j |a_j|^{1/j}$. 
\end{lem}

\begin{cor}\label{divides}
If $\sD_{\kk}$ is principal then  $\sD_{l} \subset \sD_{\kk}$ for $l \ge \kk$. 
\end{cor} 

\begin{cor}\label{zeroset}
The zero set of $\sD_{m}$ equals $ \{x; |\{\xi_1(x), \ldots, \xi_n(x)\}| < \kk\}$. 
\end{cor}

\begin{cor}\label{firstblowup}
Suppose that $a_1= 0$.   Then $\sD_{2}$ is principal if and only if so is 
$(a_2^{N!/2}, \ldots , a_n^{N!/n})$, and then both ideals coincide.   

Thus, if $\sD_{2} = (g)$ then $g$ divides each $a_i^{N!/i}$ and there is $i_0$ such that 
$|g| \sim |a_{i_0}^{N!/i_0}|$.  
\end{cor} 

\begin{proof}
Shortly speaking, this corollary follows from the fact that, by Proposition \ref{firststep}, the ideals $\sD_{2}$ and 
$(a_2^{N!/2}, \ldots , a_n^{N!/n})$  have the same integral closure.  

More precisely, in the Tschirnhausen case $a_1= 0$, 
\begin{gather*}
		\max_j |\xi_j| = \max_j |\frac{1}{n} \sum_k (\xi_j-\xi_k)| \le \max_j \frac{1}{n} \sum_k |\xi_j-\xi_k| 
		\le  \max_{k\ne j} |\xi_j-\xi_k|, \\
		\max_{k\ne j} |\xi_j-\xi_k| \le \max_{k\ne j} (|\xi_j| + |\xi_k|) \le 2 \max_j |\xi_j|, 
\end{gather*}
and hence, by Lemma \ref{lem:boundedroots},
\begin{align}\label{similarities}
\size_{2}  \sim  \max_{i}  |\xi_i |^2 \sim \max_j |a_j|^{2/j}.
\end{align}  

Thus if $a_i^{N!/i}$ generates $(a_2^{N!/2}, \ldots , a_n^{N!/n})$  then for any $f\in \sD_{2}$, 
$f/ a_i^{N!/i}$ is bounded and hence holomorphic.  Therefore  $a_i^{N!/i}$ generates $\sD_{2}$.  

Conversely, if $g$ generates  $\sD_{2}$ then by \eqref{similarities} one of the $a_i^{N!/i}/g$ does not 
vanish at the origin and hence $g\in (a_2^{N!/2}, \ldots , a_n^{N!/n})$.  
\end{proof}


\section{Convexity}  \label{convexity}

For a power series in one variable $t$, $\lambda \in \C\{t\}$,  
we define its order $\ord \lambda$ as the leading exponent 
$\lambda(t) = a_0t^{\ord \lambda} + \cdots$ and set $\ord 0 := \infty$.  
Given power series $\lambda_i  \in \C\{t\}$,  $i= 1, \ldots , n$,  we define 
for $2\le \kk \le n$
$$
\alpha (\kk) : = \min_{|I|=\kk} \ord \prod_{i\ne j\in I} (\lambda_i - \lambda_j).  
$$

\begin{prop}\label{propconvex}
For $3\le \kk\le n-1$
\begin{align}\label{convex}
2\alpha (\kk) + \alpha (2) \le \alpha (\kk-1) + \alpha (\kk+1) .
\end{align}
\end{prop}

\begin{example}
If all $\lambda_i - \lambda_j$ have the same order, say equal to $1$, then $\alpha(\kk) = 2 \binom \kk 2$ 
and we have equality.  
\end{example}

For the proof we first make some reduction.  By shifting all $\lambda_i$ by $\frac 1 n \sum_i \lambda_i$ 
 we may assume that 
$\sum_i \lambda_i =0$. Then $\alpha (2) = 2 \min_i \ord  \lambda_i $. Indeed, 
\begin{align*}
		\alpha(2)/2 &= \min_{i\ne j} \ord (\la_i -\la_j) \ge \min_{i\ne j} \min\{\ord \la_i,\ord \la_j\} 
		= \min_i \ord \la_i\\
	    &= \min_i \ord \frac{1}{n} \sum_k (\la_i-\la_k) 
		\ge  \min_i \min_{i\ne k} \ord  (\la_i-\la_k) = \alpha(2)/2.
\end{align*}
Thus,
by dividing all $\lambda_i$ 
by $t^{\alpha (2)/2}$, we may suppose that $\alpha (2) =0$.  Then \eqref{convex} becomes a 
genuine convexity relation.    

Let us  divide the set of $\lambda_i$ into the union of two disjoint non-empty subsets 
$\{ \lambda_i; i\in I\}$ and  $\{ \lambda_j; i\in J\}$, $|I|= n_1<n$, $|J|=n_2<n$, $n_1+n_2=n$, so that 
for each $i\in I, j\in J$, $\ord (\lambda_i - \lambda_j) =0$.  We shall call such a partition 
\emph{a splitting}.  

The corresponding orders for these two families we denote by $\beta(\kk)$, $2\le \kk \le n_1$, and 
$\gamma (\kk) $, $2\le \kk \le n_2$.   Then for each $2\le \kk\le n$ there is $0\le \kk_1 \le n_1$  
such that 
\begin{align}\label{sum}
\alpha (\kk) = \beta (\kk_1) + \gamma(\kk-\kk_1).
\end{align}
It is possible that $\kk_1$ or $\kk -\kk_1$ is equal to 
$0$ or $1$. In this case we put $\beta (0)=\beta (1) =\gamma (0)=\gamma (1) =0$.   

For a couple of integers  $a\le b$ we denote by $[a \ .. \ b]:= \{c\in \Z; a\le c\le b\}$ the set of all integers between $a$ and $b$
and call it an \emph{interval}. 

\begin{prop}\label{propsplitting}  For each $\kk$ the  set of $\kk_1$ such that \eqref{sum} holds is an interval.  If we denote this interval  by $[\lk (\kk) \ .. \ \rk (\kk)]$ then, 
for $2\le \kk \le n-1$, $\lk (\kk) \le \lk (\kk+1) \le \lk (\kk)+1$ and $\rk  (\kk) \le \rk  (\kk+1) \le \rk (\kk)+1$.  
\end{prop}

\begin{proof} [Proof of Propositions \ref{propconvex} and \ref{propsplitting}]

First we show Proposition \ref{propsplitting} assuming Proposition \ref{convex} 
for $\beta$ and $\gamma$.  Let 
$$\varphi_\kk (\kk_1) : =  \beta (\kk_1) + \gamma(\kk-\kk_1).$$
The set of $\kk_1$ such that \eqref{sum} holds is the set of $\kk_1$ at which 
$\varphi_\kk (\kk_1)$ 
is minimal. It is an interval since, by assumption, $\varphi_\kk (\kk_1) $ is convex.   
Moreover, by convexity, $\varphi_\kk$ is decreasing on $[2 \ .. \  \lk(\kk)]$ and increasing on 
$[\rk (\kk) \ .. \ n_1]$.  

We show that it is not possible that $\varphi_\kk(\kk_1) \le \varphi_\kk(\kk_1+1)$ and 
 $\varphi_{\kk+1} (\kk_1+1) \ge \varphi_{\kk+1} (\kk_1+2)$ with one of these inequalities strict. 
  Indeed, if this were 
 the case then  
 \begin{align*}
& \varphi_\kk(\kk_1) + \varphi_{\kk+1} (\kk_1+2)  = \beta(\kk_1) + \gamma (\kk-\kk_1) +  \beta(\kk_1+2) + \gamma (\kk-\kk_1-1 )  \\
&< \varphi_\kk(\kk_1+1) + \varphi_{\kk+1} (\kk_1+1) =  2 \beta(\kk_1+1) + \gamma (\kk-\kk_1-1) +   \gamma (\kk-\kk_1 ) 
 \end{align*}
 which contradicts the convexity of $\beta$.  This implies that  $\rk  (\kk+1) \le \rk (\kk)+1$ and $  \lk (\kk+1) \le \lk (\kk)+1$.
 By interchanging $I, \beta$ and $J, \gamma$ we obtain $\overline K_2(\kk+1) \le \overline K_2(\kk)+1$ and 
	$\underline K_2(\kk+1) \le \underline K_2(\kk)+1$ which are the desired inequalities in view of $\overline K_2(\kk) = \kk -\lk(\kk)$ and 
	$\underline K_2(\kk) = \kk - \rk(\kk)$, where $\underline K_2$, $\overline K_2$ play the roles of 
	$\underline K_1$, $\overline K_1$ for $\ga$. 

 Now we show how Proposition  \ref{propsplitting}  implies Proposition \ref{propconvex} for $\al$.  Fix 
 $\kk$ such that $3\le \kk\le n-1$.  By Proposition   \ref{propsplitting} we may assume 
 $[\lk (\kk-1) \  .. \ \rk (\kk-1)] \cap [\lk (\kk) \  .. \ \rk (\kk)]  \ne \emptyset$.  
 (If this is not the case then $\lk (\kk)=\rk (\kk) = \lk (\kk-1)+1 = \rk (\kk-1)+1$ and thus 
 $\underline K_2(\kk-1) = \overline K_2(\kk-1) = \underline K_2(\kk) = \overline K_2(\kk)$.)  
 
 Fix $\kk_1 \in [\lk (\kk-1) \  .. \ \rk (\kk-1)] \cap [\lk (\kk) \  .. \ \rk (\kk)] $.  Then,  
 by Proposition  \ref{propsplitting}, either  $\kk_1$ or $\kk_1+1$ belongs to $ [\lk (\kk+1) \  .. \ \rk (\kk+1)] $.  
 In the former case 
  \begin{align*}
 \alpha (\kk -1)+ \alpha (\kk +1) &= \beta(\kk _1) + \gamma (\kk - \kk_1-1) +  \beta(\kk_1) + \gamma (\kk-\kk_1+1 )  \\
&\ge  2( \beta(\kk_1) + \gamma (\kk-\kk_1)) = 2\alpha (\kk) .   
 \end{align*}
 In the latter case 
  \begin{align*}
\alpha (\kk-1)+  \alpha (\kk+1) &= \beta(\kk_1) + \gamma (\kk-\kk_1-1) +  \beta(\kk_1+1) + \gamma (\kk-\kk_1 )  \\
&\ge  2( \beta(\kk_1) + \gamma (\kk-\kk_1)) = 2\alpha (\kk),  
 \end{align*}
 since $ \beta(\kk_1+1) + \gamma (\kk-\kk_1-1)\ge  \beta(\kk_1) + \gamma (\kk-\kk_1) $  as $\kk_1 \in  [\lk (\kk) \  .. \ \rk (\kk)] $.  
\end{proof}


\medskip

\section{Splitting}\label{splitting}

We suppose that  $a_i \in \C\{x\}$, $x=(x_1, \ldots , x_r)$, and that $P_a$ factors  
$$
P_a(Z) = P_b(Z) P_c (Z),
$$
where $P_b(Z) = Z^{n_1} + b_1 Z^{n_1-1} + \cdots + b_{n_1}$, 
 $P_c(Z) = Z^{n_2} + c_1 Z^{n_2-1} + \cdots + c_{n_2}$, $b_i, c_j \in \C\{x\}$, $n=n_1+n_2$, $n_1,n_2>0$.  
 
 \medskip
 \noindent
 \textbf{Assumption.} 
 We shall assume that the resultant of $P_b$ and $P_c$ does not vanish at $0$, that is 
 $P_{b(0)}$ and $P_{c(0)}$ do not have common roots.  
 
 \medskip
   In order to distinguish the ideals $\sD_{\kk}$ for 
 polynomials $P_a, P_b$, and $P_c$ we shall denote them by $\sD_{a,\kk}$, $\sD_{b,\kk}$, and 
 $\sD_{c,\kk}$ respectively; likewise for the size functions  $\size_{a,\kk}$, $\size_{b,\kk}$, 
 $\size_{c,\kk}$.

 \begin{prop}\label{firstsplitting} \hfil
\begin{enumerate}
\item
 Suppose that $\sD_{a,\kk}$ is a principal ideal.  Then there are $\kk_1\ge 0$, $\kk_2\ge 0$, 
 $\kk_1 + \kk_2 = \kk$, such that 
 \begin{align}\label{polsplitting}
\size_{a,\kk}  \sim \size_{b,\kk_1} \size_{c,\kk_2}, 
 \end{align}
  and for any such $\kk_1, \kk_2$, 
  $\sD_{b,\kk_1}$ and $\sD_{c,\kk_2}$ are principal.  
  (We put $S_0=S_1=1$.)  Moreover, the set of those $\kk_1$ for which 
  \eqref{polsplitting} holds, with $\kk_2=\kk-\kk_1$,   
  is an interval.  We denote this interval by $[ \lk (\kk) \ .. \  \rk (\kk)]$.  
 \item
If  the ideals $\sD_{a,\kk}$  and $\sD_{a,\kk+1}$ are principal then $\lk (\kk) \le \lk (\kk+1) \le \lk (\kk)+1$ and 
$\rk  (\kk) \le \rk  (\kk+1) \le \rk (\kk)+1$.    
 \item
 Suppose that the ideals $\sD_{a,i}$ are principal for $2 \le i \le \kk$ and that \eqref{polsplitting} holds. 
 Then $\sD_{b,i}$ and $\sD_{c,j}$ are principal for all $i \le \kk_1$ and $j \le \kk_2$.
\end{enumerate}
 \end{prop}
 
 \begin{proof}
 For fixed $x$ we may order the roots of $P_{a(x)}$.  Given $I\subset \{1, \ldots , n\}$, $|I|=\kk$, we divide $I=I' \cup I''$, so that $I'$, $|I'|=\kk_1$,  labels the roots of $P_{b(x)}$ and $I''$, $|I''|=\kk_2$,  the roots of $P_{c(x)}$.  Then
 $$
 \prod_{i,j \in I, i\ne j} (\xi_i (x) -\xi_j(x)) = \prod_{i\ne j \in I'} (\xi_i (x) -\xi_j(x))  \cdot  \prod_{i\ne j \in I''} (\xi_i (x) -\xi_j(x)) \cdot \varphi(x), 
 $$
 and $\varphi (x) =  \prod_{i \in I', j\in I''} (\xi_i (x) -\xi_j(x))$ is non-zero and   
 close to $\prod_{i \in I', j\in I''} (\xi_i (0) -\xi_j(0))$.  
Therefore,  as functions of $x$,   
\begin{align}\label{relation}
\size_{a,\kk}  \sim \max_{\kk_1+\kk_2=\kk} \size_{b,\kk_1} \size_{c,\kk_2} .
\end{align}

Suppose that $\sD_{a,\kk}$ is generated by $f$.  By \eqref{relation}  there exist $\kk_1+ \kk_2= \kk$ and 
$h\in \sD_{b,\kk_1}$, $g\in \sD_{c,\kk_2}$ such that 
$(gh/f)(0)\ne 0$.  Then  \eqref{polsplitting} holds   
 and, by Proposition \ref{criterion}, $g$ generates $ \sD_{b,\kk_1}$ and 
 $h$ generates $ \sD_{c,\kk_2}$.

 Thus if  $\sD_{a,\kk}$ is principal  then the set of $\kk_1$ for which \eqref{polsplitting} holds, which we denote by $\mathcal M$,  
 is non-empty.  By the curve selection lemma, see for instance \cite{Milnor68},    
  \eqref{polsplitting} holds  
 if and only if it holds on every real analytic curve $x(t) : (\R,0) \to (\C^r,0)$.  
  Therefore, by  \eqref{polsplitting} and  Proposition~\ref{propsplitting}, 
$\mathcal M$, as  an intersection of intervals, 
 is an interval.  Denote it by $[ \lk (\kk) \ .. \  \rk (\kk)]$.  Moreover, if
  $\sD_{a,\kk}$  and $\sD_{a,\kk+1}$ are principal then
  $\lk (\kk) , \rk (\kk)$ satisfy the inequalities of Proposition \ref{propsplitting}.  Indeed,  suppose for instance that  $\rk (\kk) \le \rk (\kk+1)$ fails.  Then there exists a real analytic curve on which 
 $\size_{a,\kk+1}  \sim  \size_{b,\kk_1} \size_{c,\kk+1 - \kk_1}$ with $\kk_1=  \rk (\kk) $  fails but is satisfied for 
 some $\kk_1< \rk (\kk) $.    
 But this contradicts  Proposition \ref{propsplitting} since \eqref{polsplitting} holds on this curve.  This shows \thetag{1} and \thetag{2}.

Finally 
 \thetag{3} follows from \thetag{1} and \thetag{2}. 
 \end{proof}

 \begin{prop}\label{secondsplitting}
 Suppose that $\sD_{a,i}$ are principal for all $2\le i\le n$.  Let $\kk_1(\kk)$, $\kk_2(\kk) = \kk- \kk_1(\kk)$ 
 be two non-decreasing integer-valued functions defined for $0 \le \kk \le n$, such that  $0\le \kk_1(\kk)\le n_1$, 
 $0\le \kk_2(\kk)\le n_2$, for  every $0\le \kk\le n$, and such that  \eqref{polsplitting} holds 
 for each triple $(\kk, \kk_1(\kk), \kk_2(\kk))$, $0\le \kk\le n$.  
 Let $\sD_{a,\kk}=(f_\kk) $, $\sD_{b,\kk_1}=(g_{\kk_1})$, 
 $\sD_{c,\kk_2}=(h_{\kk_2}) $.  
 
 Then, for every $1\le \kk\le n$ exactly one of the following two cases happens
\begin{enumerate}
\item
$\kk_1(\kk)= \kk_1(\kk-1) +1 $, $\kk_2(\kk)= \kk_2(\kk-1)$. 
 \item
$\kk_2(\kk)= \kk_2(\kk-1) +1 $, $\kk_1(\kk)= \kk_1(\kk-1)$. 
\end{enumerate}
Moreover, if \thetag{1} holds then $g_{\kk_1} | f_\kk | g_{\kk_1}^{\kk}$, and symmetrically, if \thetag{2} holds then 
$h_{\kk_2} | f_\kk | h_{\kk_2}^{\kk}$.  
 \end{prop}
 
 \begin{proof}
Two non-decreasing non-negative functions $\kk_1, \kk_2$ such that $\kk_1(\kk) +\kk_2(\kk)=\kk$ must satisfy 
either (1) or (2).  
Thus it suffices to show that if (1) is satisfied then  $g_{\kk_1} | f_\kk | g_{\kk_1}^{\kk}$.  
The first condition $g_{\kk_1} | f_\kk $  follows easily from the proof of Proposition \ref{firstsplitting}.  If  
\eqref{polsplitting} holds then $g_{\kk_1} h_{\kk_2}$ generates $\sD_{a,\kk}$ and we may 
 suppose that 
 \begin{align}\label{first}
 f_\kk =  g_{\kk_1} h_{\kk_2} . 
 \end{align} 
 By Corollary \ref{divides}, $1= f_1|f_2| \cdots | f_n$, and similarly   $1=g_1|g_2| \cdots | g_{n_1}$ 
 and  $1=h_1|h_2 |\cdots | h_{n_2}$.  
 
Let $r$ be given by  $\kk_1(\kk)= \kk_1(\kk-1) +1 = \kk_1 (\kk-2) +1 = \cdots = \kk_1(\kk-r) + 1 = \kk_1(\kk-r -1) + 2$; 
 we have Case (1) for $m$ and  $m-r$ and Case (2) in between.   
We write $\kk_1 = \kk_1(\kk), \kk_2 = \kk_2(\kk)$ for short.   Then 
  \begin{align*}
  &  f_{\kk-1}  = g_{\kk_1-1} h_{\kk_2} | g_{\kk_1} h_{\kk_2-1}  \\ 
    &\cdots   \\ 
   &  f_{\kk-r+1}  = g_{\kk_1-1} h_{\kk_2-r+2} | g_{\kk_1} h_{\kk_2-r+1}  \\ 
     &  f_{\kk-r}  = g_{\kk_1-1} h_{\kk_2-r+1 }   .  
 \end{align*}
  Hence 
 $$
 h_{\kk_2} | (  {g_{\kk_1}}/ {g_{\kk_1-1}}  )  h_{\kk_2-1} | \cdots | (  {g_{\kk_1}}/ {g_{\kk_1-1}}  ) ^{r-1} h_{\kk_2-r+1} .
 $$
 Consequently 
  $$
 f_\kk = g_{\kk_1} h_{\kk_2} | (  {g_{\kk_1}}/ {g_{\kk_1-1}}  ) ^{r} g_{\kk_1-1}  h_{\kk_2-r+1} 
 =( {g_{\kk_1}}/ {g_{\kk_1-1}}  ) ^{r} f_{\kk-r}.  
 $$
Since  $\kk_1(\kk-r) + 1 = \kk_1(\kk-r -1) + 2$, by induction on $\kk_1$,  $ f_{\kk-r} | g_{\kk_1-1}^{\kk-r }$  (we may start the induction by formally putting $f_1=g_1=h_1=1$), which shows $ f_\kk | g_{\kk_1}^{\kk}$ as $g_{\kk_1-1}|g_{\kk_1}$.   
 \end{proof}
 
 \begin{rem}\label{kbkc}
By Proposition \ref{firstsplitting}, both  $\kk_1(\kk) = \lk (\kk)$ and $\kk_1(\kk) = \rk (\kk)$ satisfy the assumptions 
 of Proposition \ref{secondsplitting} if we complete them by putting $\lk (1)=0$, $\rk (1)=1$.  In particular, there exists a function $\kk_1 (\kk)$ satisfying the assumptions 
 of Proposition \ref{secondsplitting}, such  that $\kk_1(2) = \kk_2(2) =1$. 
This follows from the assumption that $P_{b(0)}$ and $P_{c(0)}$ do not have common roots;
 in particular, not all roots of $P_{a(0)}$ coincide and thus $\size_{a,2} \sim 1$.  
 
Fix such a function $\kk_1 (\kk)$.  For $2\le \kk_1 \le n_1$  define $\kk_b(\kk_1)$ as the smallest $\kk$ such that 
$\kk_1= \kk_1(\kk)$, i.e. $\kk_b(\kk_1)$ as a function is the lowest inverse of $\kk_1(\kk)$.  Similarly we define   
$\kk_c(\kk_2)$ for $2\le \kk_2\le n_2$.  The functions $\kk_b$ and $\kk_c$ are strictly increasing and each $3\le \kk \le n$ is in the image of precisely one of them.  See the example in Figure~\ref{fig1}.
\end{rem}

\begin{figure}[ht]
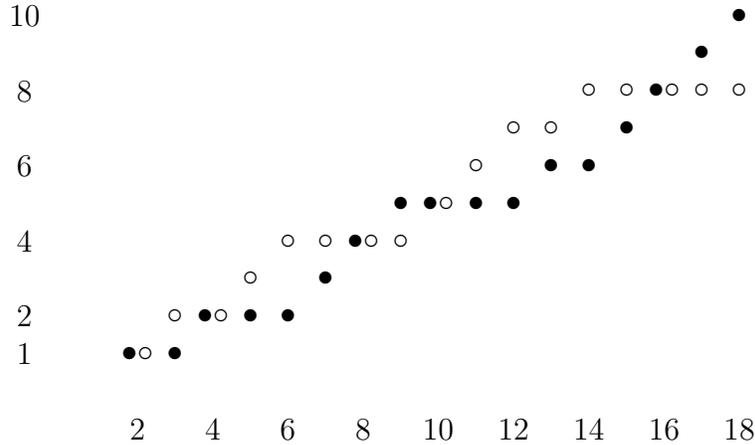

\[
	\xygraph{
	!{<0cm,0cm>;<.5cm,0cm>:<0cm,.5cm>::}
!{(2,1) }*+{\bullet\circ}                    
!{(3,1) }*+{\bullet} !{(3,2) }*+{\circ} 
!{(4,2) }*+{\bullet \circ}
!{(5,2) }*+{\bullet} !{(5,3) }*+{\circ} 
!{(6,2) }*+{\bullet} !{(6,4) }*+{\circ}
!{(7,3) }*+{\bullet} !{(7,4) }*+{\circ}
!{(8,4) }*+{\bullet\circ} 
!{(9,5) }*+{\bullet} !{(9,4) }*+{\circ}
!{(10,5) }*+{\bullet\circ} 
!{(11,5) }*+{\bullet} !{(11,6) }*+{\circ}
!{(12,5) }*+{\bullet} !{(12,7) }*+{\circ}
!{(13,6) }*+{\bullet} !{(13,7) }*+{\circ}
!{(14,6) }*+{\bullet} !{(14,8) }*+{\circ}
!{(15,7) }*+{\bullet} !{(15,8) }*+{\circ}
!{(16,8) }*+{\bullet\circ} 
!{(17,9) }*+{\bullet} !{(17,8) }*+{\circ}
!{(18,10) }*+{\bullet} !{(18,8) }*+{\circ}
!{(19,10) }*+{\bullet} !{(19,9) }*+{\circ}
!{(-1,1) }*+{1}
!{(-1,2) }*+{2}
!{(-1,4) }*+{4}
!{(-1,6) }*+{6}
!{(-1,8) }*+{8}
!{(-1,10) }*+{10}
!{(2,-1) }*+{2}
!{(4,-1) }*+{4}
!{(6,-1) }*+{6}
!{(8,-1) }*+{8}
!{(10,-1) }*+{10}
!{(12,-1) }*+{12}
!{(14,-1) }*+{14}
!{(16,-1) }*+{16}
!{(18,-1) }*+{18}
	}
\]	
\caption{An example for the functions $\kk_1(\kk)$, bullets $\bullet$, and $\kk_2(\kk)$, circles $\circ$. 
In this example $(\kk_b(2),\kk_b(3),\ldots) = (4,7,8,9,13,15,16,17,18)$ and 
$(\kk_c(2),\kk_c(3),\ldots) = (3,5,6,10,11,12,14,19)$; these are the points where the respective sequence
increases by $1$. The two sets form a partition of the integers between $3$ and $19$.} 
\label{fig1}
\end{figure}



\section{Proof of Theorem \ref{roots}} \label{proofofmaintheorem}

\begin{proof}
We first put $P_a$ into Tschirnhausen form \eqref{Tschirnhausen}, then split it into two factors
 after blowing up $\sD_{2}$.  These factors are defined locally on the blow-up space $M_2$.  Then we 
 put each of  these factors into Tschirnhausen form and continue the process  
by splitting the subsequent factors on $M_3$, $M_4$ and so on, putting first every new factor into Tschirnhausen form.    
 At each stage the Tschirnhausen transformation shifts the roots exactly by the term  
 $A_\kk \varphi_\kk $ of \eqref{sums}.  
 
 In order to simplify the notation we  use the same letters for the functions and their pull-backs  to  the blow-up spaces, for instance we write $a_k$ for $a_k\circ \sigma _\kk$.  

\medskip
\noindent
\textbf{Step 1.}
First we perform  the Tschirnhausen transformation by replacing $Z$ by $Z- a_1/n$ 
and hence shifting the roots by  $a_1/n$.  Thus we put  
$$
\varphi _1: = a_1 \quad A_1 := \frac 1 n.
$$  (Recall that the Tschirnhausen transformation does not change the ideals $\sD_{\kk}$.)  
After this transformation we may assume that $P_a (Z)$ is in the Tschirnhausen form \eqref{Tschirnhausen}.  

\medskip
\noindent
\textbf{Step 2.}
Fix $p_2\in M_2$ and a privileged system of coordinates $y_{2,1}, \ldots , y_{2,n}$ at $p_2$.  
 We shall split $P_a$   at $p_2$ 
\begin{align}\label{splittingofP_a}
P_a = P_b P_c ,
\end{align}
where $P_b(Z) = Z^{n_1} + \sum_{i=1}^{n_1}  b_i Z^{n_1-i}$,  
$P_c(Z) = Z^{n_2} + \sum_{j=1}^{n_2}  c_j Z^{n_2-j}$,  with   $n_1>0$,
$n_2>0$, $b_i, c_j \in \C\{y_{2,1}^{1/q_2}, \ldots , y_{2,n}^{1/q_2}\}$, $q_2 = N!$, as follows.  
Let $f_2$ generate $\sD_{2}$ at $p_2$.   
By Corollary \ref{firstblowup}, $f_2$ divides each $a_i^{N!/i}$ and there is $i_0$ such that 
$f_2$ equals $a_{i_0}^{N!/i_0}$ times a unit.  
Consider an auxiliary polynomial  
\begin{align}\label{step2}
  Q_{\bar a}(\tilde Z) :=  f_2^{- n / q_2} P_a (f_2^{1/q_2} \tilde Z) 
  = \tilde Z^n+   \bar a_2 \tilde Z^{n-2}+\cdots+ \bar a_n, 
\end{align}
where $\bar a_i =  f_2^{- i/ q_2 } a_i $, $\bar a_{i_0}(p_2)\ne 0$.  
Because $f_2$ is a normal crossing in $y_{2,i}$, cf.  \eqref{normalcrossings},  
$$f_2^{ 1/ {q_2} } \in  \C\{(y_{2,1}^{1/q_2}, \ldots , y_{2,r_2}^{1/q_2}, y_{2,r_2+1}, \ldots , y_{2,n}) \}.$$

We first split $Q_{\bar a}(\tilde Z) = Q_{\bar b}(\tilde Z) Q_{\bar c} (\tilde Z)$ using the following 
 lemma, see e.g.\ \cite{AKLM98} or \cite{BM90}.

\begin{lem} \label{split}
Let $Q_{a'} = Q_{b'} Q_{c'}$, $a'=(a'_1, \ldots ,a'_n)\in \C^{n}$,  $b'=(b'_1,\cdots, b'_{n_1})\in \C^{n_1}$,  
$c'=(c'_1,\cdots, c'_{n_2})\in \C^{n_2}$,  be monic complex polynomials.  Suppose that $Q_{b'}$ and $Q_{c'}$ 
have no  common root.  Then there are complex analytic mappings  $b (a)$, $c(a)$, defined in a neighborhood 
of $a'$ in $\C^n$,  such that 
 $$Q_a = Q_{b(a)}  Q_{c(a)}, $$
 and $b'= b(a'), c'=c(a')$.  
\end{lem}

\begin{proof}
If we write $ Q_a = Q_{b}  Q_{c}$ and compute $a$ as a function of $b$ and $c$, denoted by $a(b,c)$, then the 
Jacobian determinant of $a(b,c)$ equals the resultant of $Q_b$ and $Q_c$ which is nonzero by assumption.  
Thus the lemma follows from the Inverse Function Theorem.  
\end{proof}

Since $\bar a_{i_0} (p_2) \ne 0$ and $Q_{\tilde a} (\tilde Z) $ is in Tschirnhausen form,  
$Q_{\bar a(p_2)} (\tilde Z) = \tilde Z^n+ \bar a_2(p_2)  \tilde Z^{d-2}+\cdots+\bar a_n (p_2) $ has at least two 
distinct complex roots and thus can be written as the product of two factors with no common roots. Then 
Lemma \ref{split} allows us to extend this splitting to a neighborhood of $p_2$  
\begin{align}\label{splitQ}
Q_{\bar a}(\tilde Z) = Q_{\bar b} (\tilde Z) Q_{\bar c} (\tilde Z).
\end{align}
 This splitting induces a splitting \eqref{splittingofP_a} of 
$P_a $ by setting $P_b(Z)  :=  f_2^{n_1 / q_2} Q_{\bar b} (f_2^{-1/q_2}Z) $, $n_1 = \deg  Q_{\bar b}$,  that is 
\begin{equation} \label{eq:3}
    b_i : =  f_2^{i/q_2}\eta_i \big(y_{2,1}^{1/q_2}, \ldots , y_{2,r_2}^{1/q_2}, y_{2,r_2+1}, \ldots , y_{2,n}\big), 
    \quad i = 1,\ldots,n_1 = \deg P_b,
  \end{equation}
  where $\eta_i$ is a convergent power series.  Similar formulas hold  for $P_c$.   The next step involves putting 
  both $P_b$ and $P_c$ in Tschirnhausen form.  Thus if we put 
  $$\varphi_2 : = b_1 = -c_1$$
   then by  the  Tschirnhausen transformation the roots of $P_b$ are shifted by $-\frac {\varphi_2}{n_1}$ and those of 
  $P_c$ by $ \frac {\varphi_2}{n_2}$.

 \medskip
\noindent
\textbf{Step 3.}  
Let $\sfC= (p_{\kk}, f_{\kk}, h_\kk, P_{\kk,i}, s_\kk , r_\kk)$ be a chain for $p_n\in M_n$ with  neighborhoods  
  $\mathcal U_{\kk} $ of $p_\kk$.  
From $\sfC$ we extract  the chains for $P_b$ and $P_c$ as follows.  
Choose the functions $\kk_1(\kk),$ $\kk_2(\kk)$, $\kk_b(\kk_1), \kk_c(\kk_2)$ as in Remark \ref{kbkc} and split the chain 
\[
  \xymatrix@R=.1cm{
 \mathcal  U_1 & \ar[l]   \mathcal  U_2 & \ar[l] \mathcal  U_3 & \ar[l] \mathcal  U_4 & \ar[l] \mathcal  U_5 & \ar[l] 
 \mathcal U_6 & \ar[l]  \mathcal U_7 & \ar[l] \cdots & P_a \\
&  &  & & & & & \\
&   & \ar[dl] \mathcal  U_{\kk_b(2)} &  & \ar[ll] \ar@{.>}[ddl] \mathcal U_{\kk_b(3)} & \ar[l] \mathcal U_{\kk_b(4)} & & \ar[ll] \cdots & P_b\\
&   \mathcal U_2 &  & & & & & \\
&   &  & \ar[ull] \ar@{.>}[uul] \mathcal U_{\kk_c(2)} &  &  & \ar[lll] \ar@{.>}[uul] \mathcal U_{\kk_c(3)} & \ar[l] \cdots & P_c\\
  }
\]

 We claim that  by Propositions \ref{firstsplitting} and  \ref{secondsplitting}  
\begin{align}\label{towerb}
 \mathcal U_2 \longleftarrow  \mathcal  U_{\kk_b(2)} \longleftarrow  \cdots \longleftarrow \mathcal  U_{\kk_b(n_1)} ,
\end{align}
defines a chain for $p_n$ associated with $P_b$.  Indeed, the pull-back of $\sD_{b,\kk_1}$ on $(\mathcal U_{\kk},p_\kk)$, $\kk=\kk_b(\kk_1)$,  
is invertible  and  generated by a (fractional) normal crossing in $y_{\kk,1}, \ldots , y_{\kk ,r_\kk}$, 
 this generator is denoted  $g_{\kk_1}$ in Proposition \ref{secondsplitting}.  The rest of the data defining the chain 
$ ( h_\kk, P_{\kk,i}, s_\kk , r_\kk)$ is the same.  

Then, if $\kk=\kk_b(\kk_1)\ge 2$, we set 
$ \varphi_{\kk}: =   \varphi_{ b,\kk_1}$.   
By Proposition \ref{secondsplitting}
$$
\varphi_{\kk}  
=   g_{\kk_1}^ {\al_{b,m_1}}    \psi_{b,\kk_1}  
=  f_{\kk}^ {\al_{b,m_1}/\kk}   \psi_{\kk} ,
$$
(this equation defines $ \psi_{\kk}$).
By the  inductive assumption  every root of $P_b$ is a combination of the $ \varphi_{b,i}$.  
Since each root of $P_a$ is either a root of $P_b$ or of $P_c$, the proof is complete.   
\end{proof}


\subsection{Proof of Lemma \ref{addendum}} \label{proofaddendum}

We prove  Lemma \ref{addendum} following closely  the steps of the proof of Theorem \ref{roots}.  

\medskip
\noindent
\textbf{Step 1.}
Clearly $\varphi_1 (t) = a_1(t)$.  Thus after a shift of $\lambda (t)$ by $\frac 1 n \varphi_1(t) $ we may assume 
that $P_a(t)$ is in Tschirnhausen form.  

\medskip
\noindent
\textbf{Step 2.}
The crucial observation is that for all $\kk \ge 2$, 
$f_\kk (a(t))$, $y_{\kk,1} (\lifta_\kk(t)), \ldots , y_{\kk, r_\kk}  (\lifta_\kk(t))$  do not vanish on $J$. 
Hence we may choose their $q_m$-th radicals continuously, 
and even of the same regularity ($C^k$, real analytic, etc.)  as the coefficients $a_i(t)$.   

Thus a root $\lambda (t)$ of $P_{a(t)}$ induces a root $f_2^{-\al_2}(t) \lambda (t) $ of $Q_{\bar a (t)}$, 
whose coefficients are now well-defined as functions of $t\in J$.    
Since the roots of $Q_{\bar b (t)}$, $Q_{\bar c (t)}$, are distinct, $f_2^{-\al_2}(t) \lambda (t) $ is a root of precisely 
one of them.  Thus we may consider $\lambda (t)$ as a root of $P_{b(t)}$ for instance.

\medskip
\noindent
\textbf{Step 3.} 
Then, on $M_\kk$, $\kk=m_b(2)$, perform  the Tschirnhausen transformation of $P_{b(t)}$, split it and 
 by choosing the radical $g_{\kk_1}^ {\al_{b,m_1}}(t)$  identify $\lambda (t)$ (shifted by the Tschirnhausen transformation) with a root of one of these factors.  We continue these procedure until 
the  last factor is of degree $1$.  

\medskip
We note that on $J$ we have $n$ everywhere distinct continuous roots  of $P_{a(t)}$.  
They separate in the above process; any two of them, shifted first  by common Tschirnhausen transformations, are roots of different factors at some stage.

\part{Example.  Roots of cubic polynomials.} 

  \section{Statement of result.}

We give a detailed presentation of the degree $3$ case as an example.  In this case 
 the resolution is explicit and the result can be made more  precise.

\begin{thm} \label{thm:P3}
  Let $I \subset \R$ be a bounded open interval. 
  Consider a monic polynomial 
  \begin{align} \label{P3}
  P(t)(Z)=Z^3+ a_1(t) Z^2 + a_2(t) Z+ a_3(t), \quad t \in I, 
  \end{align}
  with coefficients $a_j \in C^{6}(\overline I)$, $j=1,2,3$. Then:
  \begin{enumerate}
    \item If $\la_j : I \to \C$, $j=1,2,3$, denotes a continuous parameterization of the roots of $P$,
    then each $\la_j' \in L_w^{6/5}(I)$,
    in particular, each $\la_j \in W^{1,q}(I)$, for $q \in [1, 6/5)$.
    \item  
    Let $\{P_\nu; \nu \in \cN\}$ be a family of curves of polynomials  \eqref{P3} so that the set of coefficients 
$\{a_{\nu,j}; \nu \in \cN, j=1,2,3\}$ is bounded in $C^6(\overline I)$. Then the set 
$\{\la_\nu'; \la_\nu \in C^0(I) \text{ with } P_\nu(\la_\nu) = 0 \text{ on } I,\, \nu \in \cN\}$ 
is bounded in $L_w^{6/5}(I)$.
  \end{enumerate}
\end{thm}

\begin{rem}
In Theorem \ref{thm:P3} it is actually enough to require that 
$f \in C^{5,1}(\overline I,\C)$; cf.\ \cite{GhisiGobbino13}.
\end{rem}

We sketch below the proof of Theorem \ref {thm:P3}.  Thus consider $P$ after Tschirnhausen transformation  
\begin{align} \label{polynomialdeg3} 
P(Z)=Z^3+p Z+ q .
\end{align}
The discriminant of $P$ equals 
\[
\De= -27q^2 -4 p^3.
\]  

We assume that $p,q : I \to \C$ belong to $C^{6}(\overline I)$. 
By Proposition~\ref{prop:radicals}, for each $\de \in [1/6,1)$, 
\begin{align}
\begin{split}\label{spagnolo}
  & |p'(t)| \le \Lambda_{\de^{-1}} (t) |p(t)|^{1-\de} \\ 
  & |q'(t)| \le \Lambda_{\de^{-1}} (t) |q(t)|^{1-\de} \\ 
  & |\Delta'(t)| \le \Lambda_{\de^{-1}} (t) |\Delta (t)|^{1-\de}
\end{split}  
\quad \text{ a.e.\ in } I \text{ for some } \La_{\de^{-1}} \in L_w^p(I,\R), \text{ where } p = \frac{1}{1-\de}. 
\end{align}
Here we set $\La_{\de^{-1}} :=\max\{\La_{\de^{-1},p},\La_{\de^{-1},q},\La_{\de^{-1},\De}\}$.     
Note that each formula of \eqref{spagnolo} holds for every $t$ outside the zero set of $p$, $q$, or $\De$, respectively.


\section{Resolution of the discriminant}
Consider the embedded resolution of  the discriminant given by a sequence of three point blowing-ups.  We denote it by  $\sigma :M\to \C^2$.  
Note that $\sigma$ resolves also the ideal $\sI= (p^3, q^2)$, that is, makes it locally  principal and generated by a monomial.
Thus in the notation of Section \ref{sec:formulas}, $\sigma$ is a smooth principalization of 
$\sD_2$ and $\sD_3$ at the same time.  Moreover, in this case, the formulas are explicit and we do not have to use the ideal $\sK$ 
of Subsection \ref{coordinates}.

We describe the pull-back of $P$ by $\sigma$ in  the affine charts  and in each chart we give a formula for the roots of $P$.   
These formulas give the bounds on the derivative of the roots with respect to $t$.   

 \subsection{Chart 1}  $p=XY, q=Y$.  This is one of two standard charts of the blowing-up of the origin.  
 Then $\sigma^* \sI = (Y^2)$ and 
 $$
\sigma^* P(Z) = Y( \tilde Z ^3 + XY^{1/3} \tilde Z + 1) =Y \tilde P (\tilde Z) ,
 $$
 where $Z =  Y^{1/3}\tilde Z$.  The polynomial $\tilde P$ has distinct roots near the exceptional divisor $Y=0$.  
 Therefore by the Implicit Function Theorem (IFT) near the exceptional divisor  the roots are of the form 
\begin{align}\label{formula1}
 Z  = Y^{1/3} \Phi(XY^{1/3}) = q^{1/3} \Phi (p/q^{2/3}),
 \end{align}
 where $\Phi $ is an analytic function (given locally by a convergent power series).  Hence by \eqref{spagnolo}
\begin{align}\label{der1}
|Z'(t)| &\le \frac {|q'|}{|q| ^{2/3}}  | \Phi | +   \frac {|p'|}{|q| ^{1/3}}  | \Phi' |  
+   \frac {|p q'|}{|q| ^{4/3}}  | \Phi' |   \le  \Lambda_3 (t)  (| \Phi |  + C| \Phi' | )
\end{align}
taking into account that $p/q=X$ is bounded.

 \subsection{Chart 2}  We take the other standard chart of the blowing-up of the origin $p= x, q= x y$.  The pull-back of the discriminant  is not normal crossing in this chart and we have to blow up the origin again.  
 
 \subsection{Chart 2a}  $p=X, q= X^2Y$. This is one of the standard charts of the blowing-up of the origin 
 of Chart 2: $x=X, y=XY$.  
 Then $\sigma ^* \sI = (X^3)$ and 
$$
 \sigma^* P(Z) = X^{3/2}( \tilde Z ^3 +  \tilde Z + X^{1/2}Y) =X^{3/2} \tilde P (\tilde Z) ,
 $$
  where $Z =  X^{1/2}\tilde Z$.  The discriminant of $\tilde P$ equals $-(4+27 XY^2)$ and is non-zero near the exceptional divisor.   
  Therefore by the IFT  the roots of $\tilde P$ are convergent powers series in $X^{1/2}Y$ and hence the roots of $P$ are of the form 
\begin{align}\label{formula2}
 Z  = X^{1/2} \Phi(X^{1/2}Y) = p^{1/2} \Phi (q/p^{3/2}).  
 \end{align}
 Hence by \eqref{spagnolo} 
 \begin{align}\label{der2}
|Z'(t)| &\le \frac {|p'|}{2|p| ^{1/2}}  | \Phi | +   \frac {|q'|}{|p| }  | \Phi' |  +   \frac {|q p'|}{|p| ^{2}}  | \Phi' |  \le
  \frac 1 2    \Lambda_2 (t)  (| \Phi |  + C| \Phi' | )
\end{align}
taking into account that  $q/p^2=Y$ is bounded.    

 \subsection{Chart 2b}  $p=XY, q= XY^2$.  The Chart 2b is the other standard affine chart obtained by 
 blowing up the origin in Chart 2: $x=x'y', y=y'$.  Again the discriminant is  not normal crossing in this chart so we have to blow-up again.  

 \subsection{Chart 2b(i)}  $p=X^2 Y, q= X^3Y^2$.   This is one of the standard charts of blowing up the origin on Chart 2b, the other one is 2b(ii).  
On Chart 2b(i), $\sigma^* \sI = (X^6Y^3)$ and $\Delta \circ \sigma = X^6Y^3(-4-27Y)$.  Then $$
 \sigma^* P(Z) = X^3 Y^{3/2}( \tilde Z ^3 +  \tilde Z + Y^{1/2}) =X^3 Y^{3/2}  \tilde P (\tilde Z) ,
 $$
  where $Z =  XY^{1/2}\tilde Z$.  On the set where the discriminant of $\tilde P = -(4+27 Y)$  is non-zero  we may again use the IFT.  Then 
 the roots of $P$ are of the form 
 $$
 Z  = XY^{1/2} \Phi(Y^{1/2}) = p^{1/2} \Phi (q/p^{3/2}),
 $$
where $\Phi$ is a convergent power series, as on Chart 2a.  Now 
 \begin{align}\label{der2a}
   \begin{split}
     |Z'(t)| &\le \frac 1 2 \frac {|p'|}{|p| ^{1/2}}  | \Phi | 
     +   \frac {|q^{2/3}|}{|p| }  \frac { |q'|}{|q ^{2/3}| }  | \Phi' |  + \frac 3 2
       \frac {|q|}{|p^{3/2}| } 
      \frac {|p'|}{|p| ^{1/2}}  | \Phi' | \\ 
     & \le   \frac 1 2   \Lambda_2 (t)  | \Phi |  + C  (\Lambda_3 (t) + \frac 3 2 \Lambda_2(t))| \Phi' |. 
   \end{split}
\end{align}

Near $Y= - \frac {4} {27}$ we introduce a new system of coordinates 
$\tilde X = X, \tilde Y = Y+ \frac 4 {27} = 
\frac {q^2 }{p^3} + \frac 4 {27}= 
\frac {27 q^2 + 4 p^3}{27 p^3}$.  Then $Y^{1/2}$ is a convergent power series in  $\tilde Y$.     
At $\tilde X = \tilde Y=0$ the polynomial  $\tilde P$ has one single and one 
double root.  Therefore, by Lemma \ref{split} we can factorize locally 
\begin{align}\label{factorisation}
\tilde P (\tilde Z) = (\tilde Z^2 + b_1(\tilde Y) \tilde Z+ b_2 (\tilde Y)) (\tilde Z+ c(\tilde Y)), 
\end{align}
where $b_1, b_2$ and $c= - b_1$ are convergent power series in $\tilde Y$.  Thus one root of $P$ equals $- XY^{1/2} 
c(\tilde Y)$ and hence it  can be written in the form 
\begin{align}\label{formula3}
Z=  X \Phi(\tilde Y)= \frac {p^2}q \Phi \Big(\frac {q^2}{p^3}\Big)= p^{1/2} \frac {p^{3/2}}q \Phi \Big(\frac {q^2}{p^3}\Big)  
=  q^ {1/3}  \frac {p^2}{ q^{4/3}} \Phi \Big(\frac {q^2}{p^3}\Big) , 
\end{align}
where $\Phi$ is a convergent power series.  Then, taking into account that $p^3 \sim q^2$ near $\tilde X=\tilde Y=0$  
 \begin{align}\label{der3}
   \begin{split}
     |Z'(t)| &\le \frac {2|pp'|}{|q|}  | \Phi | +   \frac {|p^2 q'|}{|q^2| }  | \Phi |  
     +  3 \frac {|q|}{|p^{3/2}|} \frac {| p'  |}{|p^{1/2}| } | \Phi' |
     +  2   \frac {|q^{2/3}|}{|p|}  \frac {|q' |}{|q^{2/3}|}  | \Phi' | \\ 
     & \le  C(\Lambda_2 (t)  + \Lambda_3 (t) )  | \Phi |  + C (3 \Lambda_2 (t) + 2 \Lambda_3(t) )  | \Phi '| 
   \end{split}
\end{align}

Denote the factors of \eqref{factorisation} by $P_b$ and $P_c$.  
The discriminant 
of $P_b$ is the product of the discriminant  of $\tilde P$ that is $- 27 \tilde Y$, and an invertible 
convergent power series in $\tilde Y$.   
Indeed, this follows from the fact that the discriminant of $\tilde P$ equals the discriminant of $P_b$ 
times the square of the resultant of $P_b$ and $P_c$ which is invertible. 

Let $\tilde P_b= Z^2 +\tilde b_2$ be the Tschirnhausen transformation of $P_b$.  The roots of $\tilde P_b$ equal, up to a constant, the square root of the discriminant of $P_b$.  This gives the following form for the remaining two roots of $P$ 
\begin{align}\label{formula4}
Z=  Z_0 \pm Z_1 ,
\end{align}
where $Z_0 = - \frac 1 2 X \Phi(\tilde Y)$ is coming from 
\eqref{formula3}, and $Z_1$ is of the form 
\begin{align}\label{formula5}
Z_1=   X\tilde Y^{1/2}  \Phi_1(\tilde Y) =   \frac {p^{1/2} \Delta ^{1/2} } q \Phi_1 \Big(\frac {q^2} {p^3}\Big)
\end{align}
It remains to give a bound for $Z_1'(t)$,    
\begin{align}\label{der5}
  \begin{split}
    |Z_1'(t)| &\le  \Big(\frac {|p' |}{|p^{1/2}|}  \frac {|\Delta^{1/2} |}{|q|}    +  
 \frac {|p^{1/2}|} {|q^{1/3}|}       \frac {|\Delta ^{1/3}|  }  {|q^{2/3}|}   \frac {| \Delta'|  } {|\Delta ^{5/6}|}   +   
    \frac {|  q'| } {|q^{2/3}| } \frac {|p^{1/2}|} {|q^{1/3}|}   \frac {|\Delta^{1/2} |}{|q|}    \Big)  | \Phi_1 |  \\
      &\qquad +  \frac {|\Delta^{1/2} |}{|q|}   \Big(2 
     \frac {| q^{2/3}| } {|p|}  \frac {|  q'| } {|q^{2/3}|} 
    +   3 \frac {|q^{2}|} {|p^{3}|}   \frac {| p'|} {|p^{1/2}|} \Big) | \Phi_1' |   \\
    &\le C  ( \Lambda_2 (t) + \Lambda_6(t) +\Lambda_3(t) )  | \Phi_1 |  +  C(2\Lambda_3 (t) + 3 \Lambda_2(t) )  | \Phi_1'|.  
  \end{split}
\end{align}
We have used that $p^3 \sim q^2$ and that $\Delta /q^2$ is bounded near $\tilde X= \tilde Y=0$.

 \subsection{Chart 2b(ii)}  $p=XY^2, q= XY^3$.   Then $\sigma^* \sI = (X^2Y^6)$ and $\Delta \circ \sigma = X^2Y^6 (-4X-27)$.  We only consider the points near the origin.  The other points on this chart, including the strict 
 transform  of the discriminant, are also on Chart 2b(i) and were considered before.  On Chart 2b(ii) 
 $$
 \sigma^* P(Z) = X Y^{3}( \tilde Z ^3 + X^{1/3} \tilde Z + 1) = X Y^{3 } \tilde P (\tilde Z) ,
 $$
  where $Z =  X^{1/3}Y\tilde Z$.  Since $\tilde P$ has distinct roots near $X=Y=0$, by the IFT,  
 the roots are of the form 
 $$
 Z  = X^{1/3}Y\Phi(X^{1/3}) = q^{1/3} \Phi (p/q^{2/3}),
 $$
as on Chart 1.  Then 
\begin{align}\label{der1a}
|Z'(t)| &\le \frac {|q'|}{|q| ^{2/3}}  | \Phi | +   \frac {|p'|}{|q| ^{1/3}}  | \Phi' |  
+   \frac {|p q'|}{|q| ^{4/3}}  | \Phi' |  
 \le  \Lambda_3 (t)  (| \Phi |  + C| \Phi' | )
\end{align}
taking into account that $p/q^{2/3}=X$ is bounded.

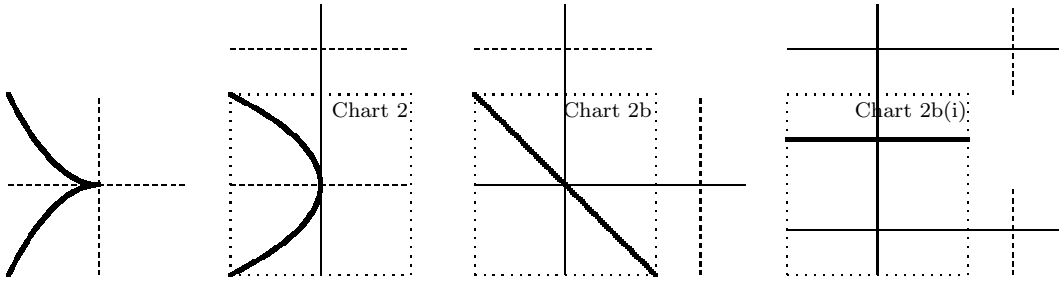
\begin{figure}[ht]
  \centering
\mbox{

\subfigure{
\setlength{\unitlength}{.3cm}
\begin{picture}(8,8)(-4,-4)
\multiput(-4,0)(0.4,0){20}{\line(1,0){0.2}}
\multiput(0,-4)(0,0.4){20}{\line(0,1){0.2}}
\linethickness{.5mm}
\qbezier(0,0)(-2,0)(-4,4)
\qbezier(0,0)(-2,0)(-4,-4)
\end{picture}
}

\subfigure{
\setlength{\unitlength}{.3cm}
\begin{picture}(9,12)(-4,-4)
\multiput(-4,0)(0.4,0){20}{\line(1,0){0.2}}
\multiput(-4,6)(0.4,0){20}{\line(1,0){0.2}}
\put(0,-4){\line(0,1){12}}
\multiput(-4,-4)(0.4,0){20}{\line(1,0){0.05}}
\multiput(-4,4)(0.4,0){20}{\line(1,0){0.05}}
\multiput(-4,-4)(0,0.4){20}{\line(0,1){0.05}}
\multiput(4,-4)(0,0.4){21}{\line(0,1){0.05}}
\put(.5,3){\tiny Chart 2}
\linethickness{.5mm}
\qbezier(0,0)(0,2)(-4,4)
\qbezier(0,0)(0,-2)(-4,-4)
\end{picture}
}

\subfigure{
\setlength{\unitlength}{.3cm}
\begin{picture}(12,12)(-4,-4)
\put(-4,0){\line(1,0){12}}
\multiput(-4,6)(0.4,0){20}{\line(1,0){0.2}}
\multiput(6,-4)(0,0.4){20}{\line(0,1){0.2}}
\put(0,-4){\line(0,1){12}}
\multiput(-4,-4)(0.4,0){20}{\line(1,0){0.05}}
\multiput(-4,4)(0.4,0){20}{\line(1,0){0.05}}
\multiput(-4,-4)(0,0.4){20}{\line(0,1){0.05}}
\multiput(4,-4)(0,0.4){21}{\line(0,1){0.05}}
\put(-.1,3){\tiny Chart 2b}
\linethickness{.5mm}
\qbezier(4,-4)(0,0)(-4,4)
\end{picture}
}

\subfigure{
\setlength{\unitlength}{.3cm}
\begin{picture}(12,12)(-4,-4)
\put(-4,-2){\line(1,0){12}}
\put(-4,6){\line(1,0){12}}
\multiput(6,-4)(0,0.4){10}{\line(0,1){0.2}}
\multiput(6,4)(0,0.4){10}{\line(0,1){0.2}}
\put(0,-4){\line(0,1){12}}
\multiput(-4,-4)(0.4,0){20}{\line(1,0){0.05}}
\multiput(-4,4)(0.4,0){20}{\line(1,0){0.05}}
\multiput(-4,-4)(0,0.4){20}{\line(0,1){0.05}}
\multiput(4,-4)(0,0.4){21}{\line(0,1){0.05}}
\put(-1,3){\tiny Chart 2b(i)}
\linethickness{.5mm}
\qbezier(4,2)(0,2)(-4,2)
\end{picture}
}

}
\caption{Bold curves represent the discriminant set and its strict transforms, thin (continuous) lines the exceptional divisors.}
\end{figure}

\subsection{Proof of Theorem \ref{thm:P3}.}
The proof follows the reasoning of  Subsection   \ref{analyticsubsection}.  

Let $p(t), q(t) \in C^6 (\overline I)$, let $\lambda (t)$ be a continuous root of \eqref{polynomialdeg3},  
and let $\Om = \{t\in I ;  (p(t),q(t)) \ne (0,0)\}$.
Then 
$(p(t),q(t)) |_{\Om}$ lifts to $M$ and Subsection   \ref{analyticsubsection} gives a bound on 
$\|\la'\|_{6/5,w, \Om}$.  Since $\lambda \equiv 0$ on $I \setminus \Om$, Theorem \ref{thm:P3} follows from  
Lemma \ref{extend}.


\medskip


\begin{thebibliography}{10}

\bibitem{AKLM98}
D.~Alekseevsky, A.~Kriegl, M.~Losik, and P.~W. Michor, \emph{Choosing roots of
  polynomials smoothly}, Israel J. Math. \textbf{105} (1998), 203--233.

\bibitem{BM90}
E.~Bierstone and P.~D. Milman, \emph{Arc-analytic functions}, Invent. Math.
  \textbf{101} (1990), no.~2, 411--424.
  
  \bibitem{BM97}
E.~Bierstone and P.~D. Milman, \emph{Canonical desingularization in characteristic zero by blowing up the maximum strata of a local invariant}, Invent. Math.
  \textbf{128} (1997),  207--302.

\bibitem{Bronshtein79}
M.~D. Bronshtein, \emph{Smoothness of roots of polynomials depending on
  parameters}, Sibirsk. Mat. Zh. \textbf{20} (1979), no.~3, 493--501, 690,
  English transl. in Siberian Math. J. \textbf{20} (1980), 347--352.

\bibitem{CJS83}
F.~Colombini, E.~Jannelli, and S.~Spagnolo, \emph{Well-posedness in the
  {G}evrey classes of the {C}auchy problem for a nonstrictly hyperbolic
  equation with coefficients depending on time}, Ann. Scuola Norm. Sup. Pisa
  Cl. Sci. (4) \textbf{10} (1983), no.~2, 291--312.

\bibitem{CL03}
F.~Colombini and N.~Lerner, \emph{Une procedure de {C}alder\'on-{Z}ygmund pour
  le probl\`eme de la racine {$k$}-i\`eme}, Ann. Mat. Pura Appl. (4)
  \textbf{182} (2003), no.~2, 231--246.

\bibitem{GhisiGobbino13}
M.~Ghisi and M.~Gobbino, \emph{Higher order {G}laeser inequalities and
  optimal regularity of roots of real functions}, Ann. Sc. Norm. Super. Pisa
  Cl. Sci. (5) \textbf{12} (2013), no.~4, 1001--1021.

\bibitem{Glaeser63R}
G.~Glaeser, \emph{Racine carr\'ee d'une fonction diff\'erentiable}, Ann. Inst.
  Fourier (Grenoble) \textbf{13} (1963), no.~2, 203--210.

\bibitem{Grafakos08}
L.~Grafakos, \emph{Classical {F}ourier analysis}, second ed., Graduate Texts in
  Mathematics, vol. 249, Springer, New York, 2008. 

\bibitem{Hartshorne77}
R.~Hartshorne, \emph{Algebraic geometry}, Springer-Verlag, New York, 1977,
  Graduate Texts in Mathematics, No. 52. 

\bibitem{Hironaka64}
H.~Hironaka, \emph{Resolution of singularities of an algebraic variety over a
  field of characteristic zero. {I}, {II}}, Ann. of Math. (2) 79 (1964),
  109--203; ibid. (2) \textbf{79} (1964), 205--326.

\bibitem{Kato76}
T.~Kato, \emph{Perturbation theory for linear operators}, second ed.,
  Grundlehren der Mathematischen Wissenschaften, vol. 132, Springer-Verlag,
  Berlin, 1976.

\bibitem{Kollar07}
J.~Koll{{\'a}}r, \emph{Lectures on resolution of singularities}, Annals of
  Mathematics Studies, vol. 166, Princeton University Press, Princeton, NJ,
  2007. 

\bibitem{KLMR05}
A.~Kriegl, M.~Losik, P.~W. Michor, and A.~Rainer, \emph{Lifting smooth curves
  over invariants for representations of compact {L}ie groups. {II}}, J. Lie
  Theory \textbf{15} (2005), no.~1, 227--234.

\bibitem{Leoni09}
G.~Leoni, \emph{A first course in {S}obolev spaces}, Graduate Studies in
  Mathematics, vol. 105, American Mathematical Society, Providence, RI, 2009.

\bibitem{Malgrange67}
B.~Malgrange, \emph{Ideals of differentiable functions}, Tata Institute of
  Fundamental Research Studies in Mathematics, No. 3, Tata Institute of
  Fundamental Research, Bombay, 1967.

\bibitem{MarcusMizel72}
M.~Marcus and V.~J. Mizel, \emph{Absolute continuity on tracks and mappings of
  {S}obolev spaces}, Arch. Rational Mech. Anal. \textbf{45} (1972), 294--320.

\bibitem{Milnor68}
J.~Milnor, \emph{Singular points of complex hypersurfaces}, Annals of
  Mathematics Studies, No. 61, Princeton University Press, Princeton, N.J.,
  1968. 
  
\bibitem{ParusinskiRond2012} 
A.~Parusi\'nski, G.~Rond, \emph{The Abhyankar-Jung Theorem}, 
  Journal of Algebra \textbf{365} (2012), 29--41.

\bibitem{ParusinskiRainerHyp}
A.~Parusi{{\'n}}ski and A.~Rainer, \emph{A new proof of {B}ronshtein's
  theorem}, J. Hyperbolic Differ. Equ. \textbf{12} (2015), no.~4, 671--688.


\bibitem{RS02}
Q.~I. Rahman and G.~Schmeisser, \emph{Analytic theory of polynomials}, London
  Mathematical Society Monographs. New Series, vol.~26, The Clarendon Press
  Oxford University Press, Oxford, 2002.

\bibitem{RainerAC}
A.~Rainer, \emph{Perturbation of complex polynomials and normal operators},
  Math. Nachr. \textbf{282} (2009), no.~12, 1623--1636. 

\bibitem{RainerQA}
A.~Rainer, \emph{Quasianalytic multiparameter perturbation of polynomials and
  normal matrices}, Trans. Amer. Math. Soc. \textbf{363} (2011), no.~9,
  4945--4977.

\bibitem{RainerN}
A.~Rainer, \emph{Perturbation theory for normal operators}, Trans. Amer. Math.
  Soc. \textbf{365} (2013), no.~10, 5545--5577. 

\bibitem{Spagnolo99}
S.~Spagnolo, \emph{On the absolute continuity of the roots of some algebraic
  equations}, Ann. Univ. Ferrara Sez. VII (N.S.) \textbf{45} (1999),
  no.~suppl., 327--337 (2000), Workshop on Partial Differential Equations
  (Ferrara, 1999).

\bibitem{Spagnolo00}
S.~Spagnolo, \emph{Local and semi-global solvability for systems of non-principal
  type}, Comm. Partial Differential Equations \textbf{25} (2000), no.~5-6,
  1115--1141.

\bibitem{Tarama00}
S.~Tarama, \emph{{On the lemma of Colombini, Jannelli and Spagnolo}}, Memoirs
  of the Faculty of Engineering, Osaka City University \textbf{41} (2000),
  111--115.

\bibitem{Wakabayashi86}
S.~Wakabayashi, \emph{Remarks on hyperbolic polynomials}, Tsukuba J. Math.
  \textbf{10} (1986), no.~1, 17--28.

\end{thebibliography}

\def\cprime{$'$}
\providecommand{\bysame}{\leavevmode\hbox to3em{\hrulefill}\thinspace}
\providecommand{\MR}{\relax\ifhmode\unskip\space\fi MR }
\providecommand{\MRhref}[2]{%
  \href{http://www.ams.org/mathscinet-getitem?mr=#1}{#2}
}
\providecommand{\href}[2]{#2}

\end{document}